\documentclass[11pt]{amsart}
\usepackage[top=1in, bottom=1in, left=1in, right=1in]{geometry}

\usepackage{mathtools,mathrsfs,amsthm,amsfonts,amssymb,graphicx,appendix,cancel,caption,subcaption}
\usepackage{xspace,enumerate,comment,bbm}
\usepackage[super]{nth}

\usepackage{tikz,pgfplots}
\usepackage{hyperref}
\usepackage[normalem]{ulem}
\usepackage[active]{srcltx}

\numberwithin{equation}{section}
\newtheorem{theorem}{\sc Theorem}[section]

\newtheorem{corollary}[theorem]{\sc Corollary}

\newtheorem{lemma}[theorem]{\sc Lemma}
\newtheorem{proposition}[theorem]{\sc Proposition}

\theoremstyle{remark}
\newtheorem{remark}[theorem]{\sc Remark}

\newcommand{\ol}[1]{\overline{#1}}
\newcommand{\ul}[1]{\underline{#1}}

\renewcommand{\t}[1]{\tilde{#1}}

\renewcommand{\epsilon}{\varepsilon}
\renewcommand{\le}{\leqslant}
\renewcommand{\ge}{\geqslant}
\newcommand{\pd}{\partial}

\newcommand{\Con}{\operatorname{C}}
\newcommand{\UC}{\operatorname{UC}}
\newcommand{\Lip}{\operatorname{Lip}}

\makeatletter
\@namedef{subjclassname@2020}{\textup{2020} Mathematics Subject Classification}
\makeatother

\begin{document}
	
\title[Loss of quasiconvexity in the periodic homogenization of viscous HJ equations]{Loss of quasiconvexity in the periodic homogenization of viscous Hamilton-Jacobi equations}

\author[E.\ Kosygina]{Elena Kosygina}
\address{Elena Kosygina\\ Department of Mathematics\\ Baruch College\\  One Bernard Baruch Way\\ Box B6-230, New York, NY 10010\\ USA}
\email{elena.kosygina@baruch.cuny.edu}
\urladdr{http://www.baruch.cuny.edu/math/elenak/}
\thanks{E.\ Kosygina was partially supported by the Simons Foundation (Award \#523625).}

\author[A.\ Yilmaz]{Atilla Yilmaz}
\address{Atilla Yilmaz\\ Department of Mathematics\\ Temple University\\ 1805 North Broad Street, Philadelphia, PA 19122, USA}
\email{atilla.yilmaz@temple.edu}
\urladdr{http://math.temple.edu/$\sim$atilla/}
\thanks{A.\ Yilmaz was partially supported by the Simons Foundation (Award \#949877).}
	
%\date{\today}
\date{September 21, 2023}

 % 35B27 (Homogenization in context of PDEs; PDEs in media with periodic structure); 35F21 (Hamilton-Jacobi equations); 35D40 (Viscosity solutions to PDEs); 60G10 Stationary stochastic processes; 60K37 Processes in random environments.

\begin{abstract}
	We show that, in the periodic homogenization of uniformly elliptic Hamilton-Jacobi equations in any dimension, the effective Hamiltonian does not necessarily inherit the quasiconvexity property (in the momentum variables) of the original Hamiltonian. This observation is in sharp contrast with the first order case, where homogenization is known to preserve quasiconvexity. We also show that the loss of quasiconvexity is, in a way, generic: when the spatial dimension is $1$, every convex function $G$ can be modified on an arbitrarily small open interval so that the new function $\t{G}$ is quasiconvex and, for some 1-periodic and Lipschitz continuous $V$, the effective Hamiltonian arising from the homogenization of the uniformly elliptic Hamilton-Jacobi equation with the Hamiltonian $H(p,x)=\t{G}(p)+V(x)$ is not quasiconvex.
\end{abstract}

\maketitle

\section{Introduction}\label{sec:intro}

This paper considers the periodic homogenization (as $\epsilon\downarrow 0$) of the equation
\begin{equation}
	\label{eq:HJint}
	\pd_t u^\epsilon=\epsilon \sigma^2\Delta_x u^\epsilon+H\left(D_x u^\epsilon,\frac{x}{\epsilon}\right),\quad (t,x)\in(0,\infty)\times \mathbb{R}^d,
\end{equation}
with $u^\epsilon(0,x)=g(x)$, $x\in\mathbb{R}^d$. We shall assume
that $\sigma\ge 0$ is a constant,
$H\in \Lip_{\text{loc}}(\mathbb{R}^d\times\mathbb{R}^d)$ is
superlinear in the momentum variables and $1$-periodic in
each of the spatial variables, and $g$ is uniformly continuous. Even
though many of the quoted below results are known for a general
degenerate elliptic operator $\text{tr}(A(x)D_x^2)$ in place of
$\sigma^2\Delta_x$, our discussion will focus on the above model viscous
case $\sigma>0$ vs.\ the inviscid case $\sigma=0$.

It is well known that, under a set of standard growth and regularity
conditions, equation \eqref{eq:HJint} homogenizes (see \cite{LPV}, \cite{Evans92} and Appendix \ref{app:super}), that is, there exists a
continuous function $\ol{H}:\mathbb{R}^d\to \mathbb{R}$ such that, for
every uniformly continuous $g$, viscosity solutions $u^\epsilon$ of
\eqref{eq:HJint} converge as $\epsilon\downarrow 0$ locally uniformly in
$(t,x)$ to the unique viscosity solution $\ol{u}$ of the effective, or
averaged, equation
\begin{equation}\label{eq:effHJint}
	\pd_t \ol{u}=\ol{H}(D_x\ol{u}),\quad (t,x)\in (0,\infty)\times\mathbb{R}^d,
\end{equation}
satisfying $\ol{u}(0,\,\cdot\,)=g(\,\cdot\,)$.

We shall be mainly interested in Hamiltonians $H(p,x)$ which, in the
addition to the above listed properties, are quasiconvex in $p$, i.e.,
for all $x\in\mathbb{R}^d$ and $\lambda\in\mathbb{R}$, the sublevel
sets $\{p\in\mathbb{R}^d:\,H(p,x)\le \lambda\}$ are convex. The
question we address below is whether $\ol{H}$ necessarily inherits the
quasiconvexity property of $H$. We also point out why this property is
important for the largely open problem of stochastic homogenization of
viscous Hamilton-Jacobi equations with quasiconvex Hamiltonians, of
which \eqref{eq:HJint} is a special case.

First of all, we recall that, if $H(\,\cdot\,,x)$ is convex for all
$x\in\mathbb{R}^d$, then $\ol{H}$ is convex (\cite{LPV}, \cite{Evans92}). Moreover, there are $\inf$-$\sup$ formulas for
$\ol{H}$ (see \cite{CIPP98, Gom02} and \cite{LS2005, KRV, LS_revisited}), from which one can also see directly that the convexity of
$H(\,\cdot\,,x)$ implies the convexity of $\ol{H}$.

If $\sigma=0$ and $H$ is only quasiconvex in $p$\footnote{and, except
  for the minimal level set, all level sets of $H$ in $p$ have empty
  interior, see \cite[(H2), p.\,763]{DS09} and \cite[(2.8),
  p.\,3424]{AS} for the precise formulation which will be in force
  throughout the discussion}, then $\ol{H}$ is also quasiconvex
(\cite{DS09}($d=1$), \cite{AS}). Moreover, the inf-sup formula for
$\ol{H}$ in \cite{LS2005} (with $A\equiv 0$) from the convex case
remains valid.  An extension of the formula from \cite{CIPP98} to the
quasiconvex case is given in \cite{Nak}. These representations
manifest the preservation of quasiconvexity in the inviscid
case. In fact, there are classes of non-quasiconvex Hamiltonians
for which the effective Hamiltonian is quasiconvex. This
quasiconvexification effect was observed in \cite{ATY_nonconvex,ATY_1d} and
thoroughly studied in \cite{QTY}.

When $\sigma>0$ and $H$ is only quasiconvex, there is no known formula
for $\ol{H}$. Our main results (see Section \ref{sec:results}) show,
first in one dimension and then in all dimensions $d\ge2$, that there are
Hamiltonians $H(p,x)=G(p)+V(x)$ with $G$ quasiconvex and $V$ periodic,
satisfying all standard growth and regularity conditions, such that
$\ol{H}$ fails to be quasiconvex. 
Moreover, when $d=1$, such a quasiconvex $G$ can be constructed
starting with any convex function in $\Con^2(\mathbb{R})$ that satisfies standard growth conditions
and then modifying it on an arbitrarily small open interval.

Our results imply that, for $\sigma>0$,
unlike in the inviscid case, none of the $\inf$-$\sup$ formulas for
$\ol{H}$ can be extended, in general, from the convex to the
quasiconvex setting. Indeed, if such a formula were to hold for
$\sigma>0$, then the effective Hamiltonian for a quasiconvex $H$ would
be necessarily quasiconvex, which, as we show in this paper, need not
be true. This surprising, at least to the authors, discovery that
adding a viscous term to the equation can lead to a loss of
quasiconvexity of $\ol{H}$ goes somewhat against the tacit expectation
that essentially all major qualitative phenomena observed in the
homogenization of inviscid Hamilton-Jacobi equations should extend
to the viscous case, albeit with the understanding that such extensions
typically would not be straightforward and will require new ideas.

Apart from showing the striking difference in the attainable
``shapes'' of $\ol{H}$ in averaging of viscous vs.\ inviscid equations
\eqref{eq:HJint} with quasiconvex $H$, our results also contribute a new
insight into the study of the more general problem of stochastic
homogenization of Hamilton-Jacobi equations. Homogenization results as
well as the preservation of convexity (for $\sigma\ge 0$,
\cite{Sou99,RT,LS2005,KRV, LS_revisited,AT14}) and quasiconvexity (for
$\sigma=0$, \cite{DS09,AS}) are also known in the general stationary
ergodic setting, namely, when $H(p,x)$ is replaced with a stationary
(with respect to the shifts in $x$) ergodic process $H(p,x,\omega)$ on
some probability space $(\Omega,{\mathcal F},\mathbb{P})$. The above
references also include $\inf$-$\sup$ formulas for $\ol{H}$ when $H$
is convex and $\sigma\ge 0$ or when $H$ is quasiconvex and $\sigma=0$.
In the stationary ergodic setting, the question as to whether viscous
Hamilton-Jacobi equations with quasiconvex Hamiltonians homogenize
currently remains open for $d\ge 2$.  Even the one-dimensional viscous
case proved to be much more difficult than the inviscid one, and the
homogenization result for quasiconvex $H$ has not yet been obtained in
the desired generality.

More precisely, for $d=1$, $\sigma>0$, and
$H(p,x,\omega)=G(p)+V(x,\omega)$ with quasiconvex superlinear $G$,
homogenization of \eqref{eq:HJint} in stationary ergodic media has been
shown in \cite{Y21b} under the additional assumption that the potential $V$ 
satisfies the so-called ``hill condition'' (see \eqref{eq:hilval}).
This condition\footnote{together with the
	analogously defined ``valley condition'', see \eqref{eq:hilval}} was introduced in
\cite{YZ19,KYZ20}, and it holds for a rich class of random potentials,
but fails when $V$ is periodic (or ``rigid'' in some other way), see
\cite[Appendix B]{DK22} for a discussion. The hill condition, in
particular, guarantees that $\ol{H}$ is quasiconvex. The comparison of
this fact with our results begs the following questions (for $\sigma>0$):
\begin{itemize}
	\item [(i)] Can one  characterize stationary ergodic media
	which preserve quasiconvexity?
	\item [(ii)] In the case when the stationary ergodic medium preserves
	quasiconvexity, do any of the $\inf$-$\sup$ formulas for $\ol{H}$
	extend from convex to quasiconvex $H$?
\end{itemize}

Even though we expect that, for $d=1$, the viscous equation
\eqref{eq:HJint} with quasiconvex superlinear stationary ergodic $H$
homogenizes without any additional assumptions on the random medium, our
results indicate that proving this conjecture might not be 
easier than proving homogenization in the general nonconvex case.

For $d=1$ and $\sigma=0$, homogenization for general nonconvex
coercive $H$ has been established in \cite{ATY_1d,Gao16}. The analogous result for
the viscous case when $H(p,x,\omega)=G(p)+V(x,\omega)$, where $G$ is
superlinear and $V$ satisfies the hill-and-valley condition, was
recently proven in \cite{DKY23+}.

It has been shown by counterexamples that, for $d\ge 2$, equation
\eqref{eq:HJint} with nonconvex superlinear $H$ considered in stationary
ergodic media can fail to homogenize if $H(\,\cdot\,,x,\omega)$ has a
strict saddle point, otherwise being standard, and the environment is
slowly mixing (\cite{Zil, FS, FFZ}). Hence, the completion of the study
of homogenization of viscous Hamilton-Jacobi equations with nonconvex
$H$ in dimension one and any general results for quasiconvex $H$ in
dimensions two and higher still remain challenging open problems.

\section{Results}\label{sec:results}

\subsection{Basic notation and standing assumptions}\label{ss:setting}

Throughout the paper, we will denote by $\Con(X)$, $\UC(X)$ and $\Lip(X)$ the sets of functions on $X$ that are continuous, uniformly continuous and Lipschitz continuous, respectively, where $X$ will be $\mathbb{R}^d$ or $[0,\infty)\times\mathbb{R}^d$.

We start with $d=1$ and consider a Hamilton-Jacobi (HJ) equation of the form 
\begin{equation}\label{eq:HJ}
	\pd_tu= \sigma^2\partial^2_{xx}u + G(\pd_x u) + V(x),\quad(t,x)\in (0,\infty)\times\mathbb{R},
\end{equation}
where $\sigma \ge 0$ is a constant, $G\in\Con(\mathbb{R})$ and it is coercive, i.e., \[ \lim_{p\to\pm\infty}G(p) = \infty, \] and $V\in\Lip(\mathbb{R})$ is a bounded function, referred to as the {\em potential}. We will be concerned with viscosity solutions of \eqref{eq:HJ}, see \cite{users,barles_book,bardi} for background.

Define
\begin{equation}\label{eq:calG}
\begin{aligned}
	\mathcal{G}_0 &= \{G\in\Con(\mathbb{R}):\,\text{$G$ is coercive and, for every $\sigma\ge 0$ and every bounded $V\in\Lip(\mathbb{R})$,}\\
	&\hspace{27mm}\text{the Cauchy problem for (2.1) is well-posed in $\UC([0,\infty)\times\mathbb{R})$}\}.
\end{aligned}
\end{equation}
In Section \ref{ss:prev}, we will recall conditions under which $G\in\mathcal{G}_0$. %(Since those sufficient conditions will not be explicitly used in our arguments, we do not give them here.)

Note that, when $\sigma = 0$, \eqref{eq:HJ} is an {\em inviscid} (i.e., first-order) HJ equation. Our main result in one dimension (Theorem \ref{thm:hicaz}) is concerned with the {\em viscous} (i.e., second-order) case, where we will take $\sigma = 1$ as there is no loss of generality in doing so.

For every $\epsilon > 0$, if $u$ is a viscosity solution of \eqref{eq:HJ}, then so is $u^\epsilon(t,x) = \epsilon u\left(\frac{t}{\epsilon},\frac{x}{\epsilon}\right)$ for the equation
\begin{equation}\label{eq:epsHJ}
	\pd_t u^\epsilon=\epsilon \sigma^2 \partial^2_{xx} u ^\epsilon+G(\pd_x u^\epsilon) + V\left(\frac{x}{\epsilon}\right),\quad (t,x)\in (0,\infty)\times\mathbb{R}.
\end{equation}
Recall from Section \ref{sec:intro} that the HJ equation \eqref{eq:epsHJ} homogenizes if there exists a continuous function $\ol{H}:\mathbb{R}\to\mathbb{R}$, called the {\em effective Hamiltonian}, such that, for every $g\in\UC(\mathbb{R})$, the unique viscosity solution $u^\epsilon$ of \eqref{eq:epsHJ} satisfying $u^\epsilon(0,\,\cdot\,) = g$ converges locally uniformly on $[0, \infty)\times\mathbb{R}$ as $\epsilon\to 0$ to
the unique viscosity solution $\ol{u}$ of the HJ equation
\begin{equation}\label{eq:effHJ}
	\pd_t \ol{u}=\ol{H}(D\ol{u}),\quad (t,x)\in (0,\infty)\times\mathbb{R},
\end{equation}
satisfying $\ol{u}(0,\,\cdot\,) = g$.

\subsection{Overview of previous results in one dimension}\label{ss:prev}

In this brief account, we will restrict our attention to the homogenization of HJ equations of the form in \eqref{eq:epsHJ} in one space dimension, put precise conditions on $G$ and $V$, and recapitulate the previously obtained results that are most relevant to our discussion.

We shall say that $G\in\mathcal{G}_1$ if there exist $\alpha_0,\alpha_1>0$ and $\eta>1$ such that
\begin{equation}\label{eq:superG}
\begin{aligned}
	&\alpha_0|p|^\eta-1/\alpha_0\le G(p)\le\alpha_1(|p|^\eta+1)\ \text{for all ${p\in\mathbb{R}}$, and}\\
	&|G(p)-G(q)|\le\alpha_1\left(|p|+|q|+1\right)^{\eta-1}|p-q|\ \text{for all $p,q\in\mathbb{R}$.}
\end{aligned}
\end{equation}
With this notation, $\mathcal{G}_1\subset\mathcal{G}_0$ (defined in \eqref{eq:calG}), see, e.g., \cite{DK17} and the references therein.

\subsubsection{Periodic homogenization}

In this setting, the potential $V$ is assumed to be $1$-periodic, i.e.,
\[ V(x+1) = V(x)\ \text{for all $x\in\mathbb{R}$}. \]

If $G\in\mathcal{G}_1$, then the HJ equation \eqref{eq:epsHJ} homogenizes. In the inviscid case ($\sigma = 0$), this result is covered by \cite{ATY_1d}. In the viscous case ($\sigma > 0$), it follows from \cite{Evans92} in combination with other known results and estimates, see Appendix \ref{app:super}.

\subsubsection{Stochastic homogenization}\label{sss:stohom}

In this setting, the potential $V$ is assumed to be a realization of a stationary ergodic stochastic process. Then, homogenization is understood in almost sure sense with respect to the underlying probability measure, and by ergodicity, the effective Hamiltonian $\ol{H}$ is constant almost surely. See, e.g., \cite{JKO} or any of the references below for details.

If $G\in\mathcal{G}_1$ is convex, then the HJ equation \eqref{eq:epsHJ} homogenizes, and the effective Hamiltonian $\ol{H}$ is convex. This result is covered by \cite{Sou99,RT} in the inviscid case and by \cite{LS2005,KRV} in the viscous case.

If we drop the convexity assumption on $G$, then the homogenization of \eqref{eq:epsHJ} is established in \cite{ATY_1d} in the inviscid case. The analogous result in the viscous case is covered by \cite{DKY23+} under the additional assumption that $V$ satisfies the {\em hill-and-valley condition}:
\begin{equation}\label{eq:hilval}
\begin{aligned}
	\text{(hill)}\qquad \mathbb{P}(V(\,\cdot\,,\omega) \ge B - h\ \text{on}\ [0,y]) > 0\ &\text{for all $h\in(0,B)$ and $y>0$, and}\\
	\text{(valley)}\qquad\mathbb{P}(V(\,\cdot\,,\omega) \le h -B\ \text{on}\ [0,y]) > 0\ &\text{for all $h\in(0,B)$ and $y>0$,}
\end{aligned}
\end{equation}
where $B := \sup\{|V(x,\omega)|:\,x\in\mathbb{R}\} < \infty$ (which is non-random by ergodicity).

If, instead of dropping the convexity assumption on $G$ altogether, we replace it by {\em quasiconvexity}, i.e., assume that
\begin{equation}\label{eq:quasicon}
	\text{the sublevel set $\{p\in\mathbb{R}:\,G(p) \le \lambda\}$ of $G$ is an interval (possibly empty) for every $\lambda\in\mathbb{R}$,}
\end{equation}
then the homogenization of \eqref{eq:epsHJ} was proved earlier in \cite{DS09} (and covered by \cite{AS}) in the inviscid case and by \cite{Y21b} in the viscous case under the hill condition in \eqref{eq:hilval}. In both cases, the effective Hamiltonian $\ol{H}$ inherits the quasiconvexity of $G$.

\subsection{Our main result in one dimension}\label{ss:main}

Recall from Section \ref{ss:prev} that, when $\sigma = 0$, in the stochastic homogenization of \eqref{eq:epsHJ} (which includes periodic $V$), if $G$ is quasiconvex, then $\ol{H}$ is quasiconvex, too. Similarly, when $\sigma = 1$, in the stochastic homogenization of \eqref{eq:epsHJ} under the hill condition in \eqref{eq:hilval} (which excludes periodic $V$), if $G$ is quasiconvex, then so is $\ol{H}$.

In light of the paragraph above, it is natural to propose the following statements:
\begin{itemize}
	\item [(i)] the stochastic homogenization of \eqref{eq:epsHJ} holds in the viscous case without the hill-and-valley condition \eqref{eq:hilval}, and 
	\item [(ii)] if $G$ is quasiconvex, then $\ol{H}$ is quasiconvex, too.
\end{itemize}
As we have said in Section \ref{sec:intro}, the first statement is an open problem.

In this paper, we show that the second statement above is false by constructing a class of counterexamples in the setting of periodic homogenization. We work with coercive $G\in\Con^2(\mathbb{R})$ that are nonconvex (but still allowed to be quasiconvex) in a prescribed way, and then construct a potential $V$ for which we show that the effective Hamiltonian $\ol{H}$ corresponding to the original Hamiltonian $G(p) + V(x)$ is not quasiconvex.

Here is our first main result, in which we tacitly assume that the HJ equation \eqref{eq:epsHJ} homogenizes. Recall that $G\in\mathcal{G}_1$ (see \eqref{eq:superG}) is a sufficient condition.

\begin{theorem}\label{thm:hicaz}
	Assume that $G\in\Con^2(\mathbb{R})\cap\mathcal{G}_0$. Fix $p_1,p_2\in\mathbb{R}$ such that $p_1 < p_2$ and
	\begin{equation}\label{eq:gecit}
		G'(p_1) < 0 < G'(p_2).
	\end{equation}
	Let $K_1 = \max\{|G'(p)|:\,p\in[p_1,p_2]\}$. Assume further that
	\begin{equation}\label{eq:Gduz}
		G''(p_1) < 0\quad\text{and}\quad G''(p_1)G'(p_2)e^{-K_1} < G''(p_2)G'(p_1)e^{K_1}.
	\end{equation}
	Then, there is a $1$-periodic potential $V\in\Lip(\mathbb{R})$ such that the effective Hamiltonian $\ol{H}$ in \eqref{eq:effHJ} that arises from the homogenization of \eqref{eq:epsHJ} {\rm (}with $\sigma = 1${\rm )} is not quasiconvex.
\end{theorem}

\begin{remark}\label{rem:ters}
	By the change of variables $(p,x) \mapsto (-p,-x)$, condition \eqref{eq:Gduz} can be replaced with
	\begin{equation}\label{eq:Gters}
		G''(p_2) < 0\quad\text{and}\quad G''(p_1)G'(p_2)e^{K_1} < G''(p_2)G'(p_1)e^{-K_1}.
	\end{equation}
	Note that if
	\begin{equation}\label{eq:resmi}
		G''(p_1) < 0\quad\text{and}\quad G''(p_2) < 0,
	\end{equation}
	then both of \eqref{eq:Gduz} and \eqref{eq:Gters} are satisfied. See Figure \ref{fig:cor} for an example.
\end{remark}

\begin{figure}
	\includegraphics[width=0.5\linewidth]{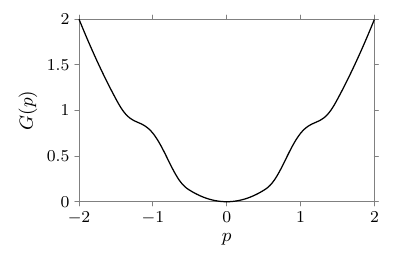}
	\caption{An example of a quasiconvex $G\in\Con^2(\mathbb{R})$ that satisfies \eqref{eq:gecit} and \eqref{eq:resmi} with $p_1 = -1$ and $p_2 = 1$.}
	\label{fig:cor}
\end{figure}

The statement of Theorem \ref{thm:hicaz} does not include \eqref{eq:quasicon} as a hypothesis, but it covers a wide class of quasiconvex $G$. In fact, we have the following corollary.

\begin{corollary}\label{cor:modify}
	For every convex $G\in\Con^2(\mathbb{R})\cap\mathcal{G}_0$ and $p_*,p^*\in\mathbb{R}$ with $p_* < p^*$, there is a quasiconvex $\t{G}\in\Con^2(\mathbb{R})\cap\mathcal{G}_0$ and a $1$-periodic potential $V\in\Lip(\mathbb{R})$ such that $\t{G}(p) = G(p)$ for all $p\not\in(p_*,p^*)$, and the effective Hamiltonian $\ol{H}$ in \eqref{eq:effHJ} that arises from the homogenization of \eqref{eq:epsHJ} {\rm (}with $\sigma = 1$ and $G$ replaced with $\t{G}${\rm )} is not quasiconvex.
\end{corollary}

\subsection{Extension to higher dimensions}

Consider the HJ equation \eqref{eq:HJint} in any dimension $d\ge2$. Assume that the Hamiltonian $H$ satisfies the standard growth and regularity conditions \eqref{eq:Hkral1}--\eqref{eq:Hkral3} in Appendix \ref{app:super}.

Recall from Section \ref{sec:intro} that, in the context of stochastic homogenization, for $\sigma\ge0$, if $H$ is convex in $p$, then the effective Hamiltonian $\ol{H}$ in \eqref{eq:effHJint} is convex, too. Similarly, for $\sigma=0$, if $H$ is quasiconvex in $p$, then so is $\ol{H}$. Our second main result states that the latter implication is not true in the periodic setting for $\sigma=1$.

\begin{theorem}\label{thm:multi}
	For every $d\ge2$, there is a quasiconvex $G\in\Con^2(\mathbb{R}^d)$ and a potential $V\in\Lip(\mathbb{R}^d)$ such that $V$ is $1$-periodic in each coordinate and the Hamiltonian $H$ defined by $H(p,x) = G(p) + V(x)$ satisfies  \eqref{eq:Hkral1}--\eqref{eq:Hkral3}, so equation \eqref{eq:HJint} {\rm (}with $\sigma = 1${\rm )} homogenizes, but the effective Hamiltonian $\ol{H}$ in \eqref{eq:effHJint} is not quasiconvex.
\end{theorem}

\section{Outline of the proofs}\label{sec:inform}

In this section, we describe the constructions and arguments that lead to the proofs of our main results without providing all of the technical details or references.

\subsection{One dimension}\label{ss:heur1d}

Assume that $G\in\Con^1(\mathbb{R})$ is coercive, the potential $V\in\Lip(\mathbb{R})$ is $1$-periodic, and the HJ equation \eqref{eq:epsHJ} homogenizes. Then, for each $\theta\in\mathbb{R}$, there exists a $1$-periodic $F_\theta\in\Con^2(\mathbb{R})$ such that
\begin{align}
	F_\theta''(x) + G(\theta + F_\theta'(x)) + V(x) &= \ol{H}(\theta)\label{eq:biz}\\
	\shortintertext{for all $x\in\mathbb{R}$. Let $f_\theta(x) = \theta + F_\theta'(x)$. Note that $f_\theta\in\Con^1(\mathbb{R})$ is $1$-periodic, $\int_0^1f_\theta(x)dx = \theta$, and}
	f_\theta'(x) + G(f_\theta(x)) + V(x) &= \ol{H}(\theta).\label{eq:atar}
\end{align}

For the sake of heuristics, formally differentiate both sides of \eqref{eq:atar} with respect to $\theta$:
\[ (\pd_\theta f_\theta)'(x) + G'(f_\theta(x))(\pd_\theta f_\theta)(x) = \ol{H}'(\theta). \]
The general solution of this derived equation is given by the formula
\begin{equation}\label{eq:sign}
	(\pd_\theta f_\theta)(x) = \left(\ol{H}'(\theta)\int_0^xe^{I_\theta(y)}dy + (\pd_\theta f_\theta)(0)\right)e^{-I_\theta(x)}
\end{equation}
for all $x\in\mathbb{R}$, where
\begin{equation}\label{eq:bured}
	I_\theta(x) = \int_0^x G'(f_\theta(y))dy.
\end{equation}
Since $(\pd_\theta f_\theta)(1) = (\pd_\theta f_\theta)(0)$, this formula yields the identity
\begin{equation}\label{eq:iden}
	\ol{H}'(\theta)\int_0^1e^{I_\theta(x)}dx = (\pd_\theta f_\theta)(0)(e^{I_\theta(1)} - 1).
\end{equation}
This identity and \eqref{eq:sign} imply that $\pd_\theta f_\theta$ does not change sign. Noting that $\int_0^1(\pd_\theta f_\theta)(x)dx = \pd_\theta\left(\int_0^1f_\theta(x)dx\right) = 1$, we deduce that $\pd_\theta f_\theta$ is positive everywhere.

Next, fix $p_1 < p_2$ such that $G'(p_1) < 0 < G'(p_2)$, and let $\displaystyle{L = \frac{G'(p_2)}{G'(p_2) - G'(p_1)} \in (0,1)}$ so that
$G'(p_1)L + G'(p_2)(1-L) = 0$. It is easy to see that, for every $\ell\in(0,L\wedge(1-L))$, there is a $1$-periodic $f\in\Con^{1,1}(\mathbb{R})$ such that $p_1 \le f(x) \le p_2$ for all $x\in\mathbb{R}$, $f(x) = p_1$ for $x\in[0,L - \ell]$, $f(x) = p_2$ for $x\in[L,1-\ell]$, and
\begin{equation}\label{eq:costco}
	\int_0^1 G'(f(x))dx = 0.
\end{equation}
Define a $1$-periodic potential $V\in\Lip(\mathbb{R})$ via the equation
\begin{equation}\label{eq:burbez}
	f'(x) + G(f(x)) + V(x) = 0,\quad x\in\mathbb{R}.
\end{equation}
Let $\theta_0 = \int_0^1 f(x)dx$. Then, $\ol{H}(\theta_0) = 0$ and $f_{\theta_0} = f$ by \eqref{eq:atar} and \eqref{eq:burbez}, and
\begin{equation}\label{eq:parpul}
	I_{\theta_0}(1) = \int_0^1 G'(f(x))dx = 0
\end{equation}
by \eqref{eq:bured} and \eqref{eq:costco}.

Assume further that $G\in\Con^2(\mathbb{R})$ and satisfies \eqref{eq:resmi}. Again for the sake of heuristics, take $x=1$ in \eqref{eq:bured} and formally differentiate both sides of this equality with respect to $\theta$:
\[ (\pd_\theta I_{\theta})(1) = \int_0^1 G''(f_\theta(y))(\pd_\theta f_\theta)(y) dy. \]
Since $f = f_{\theta_0}$ spends $1 - 2\ell$ of its ``time" at $p_1$ or $p_2$ during each period, it should follow from \eqref{eq:resmi} that
\begin{equation}\label{eq:true}
	\left.(\pd_\theta I_{\theta})(1)\right|_{\theta = \theta_0} < 0\ \text{when $\ell$ is sufficiently small.}
\end{equation}

Now go back to \eqref{eq:iden} and formally differentiate both sides of it with respect to $\theta$:
\[ 	\ol{H}''(\theta)\!\!\int_0^1e^{I_\theta(x)}dx + \ol{H}'(\theta)\,\pd_\theta\!\!\int_0^1e^{I_\theta(x)}dx = (\partial^2_\theta f_\theta)(0)(e^{I_\theta(1)} - 1) + (\pd_\theta f_\theta)(0)\,(\pd_\theta I_\theta)(1)e^{I_\theta(1)}. \]
Recall from \eqref{eq:parpul} that $I_{\theta_0}(1) = 0$. It follows from \eqref{eq:iden} that $\ol{H}'(\theta_0) = 0$, and hence, the equality above simplifies to
\[ \ol{H}''(\theta_0)\!\!\int_0^1e^{I_{\theta_0}(x)}dx = \left.(\pd_\theta f_\theta)(0)\right|_{\theta = \theta_0}\left.(\pd_\theta I_\theta)(1)\right|_{\theta = \theta_0}. \]
Therefore, if \eqref{eq:true} is indeed true, then $\ol{H}''(\theta_0) < 0$ %(since $\pd_\theta f_\theta > 0$)
when $\ell$ is sufficiently small, and $\ol{H}$ is not quasiconvex on any open set containing $\theta_0$.

In the heuristic argument that we have given above, \eqref{eq:iden} and \eqref{eq:true} play key roles. We make the following observations about them:
\begin{itemize}
	\item [(i)] The identity \eqref{eq:iden} implies that $\ol{H}'(\theta)$ and $I_\theta(1)$ have the same sign. In particular,
	\begin{equation}\label{eq:heur}
		\begin{aligned}
			&\text{if $I_\theta(1)$ is strictly positive (resp.\ negative), then for sufficiently small $h>0$,}\\
			&\text{there is a $\theta^*\in(\theta,\theta + h)$ (resp.\ $\theta^*\in(\theta - h,\theta$)) such that $\ol{H}(\theta^*) > \ol{H}(\theta)$.}
		\end{aligned}
	\end{equation}
	\item [(ii)] If \eqref{eq:true} is true, then we recall \eqref{eq:parpul} and deduce that
	\begin{equation}\label{eq:ekmek}
		I_{\theta_0 - c}(1) > 0 > I_{\theta_0 + c}(1)\ \ \text{when $c>0$ is sufficiently small.}
	\end{equation}
\end{itemize}

In Sections \ref{sec:general} and \ref{sec:construct}, we will give a rigorous version of this heuristic argument by circumventing the question as to whether we can differentiate $\ol{H}(\theta)$, $f_\theta$ and $I_\theta(1)$ with respect to $\theta$. Specifically, instead of justifying \eqref{eq:iden} and \eqref{eq:true}, we will directly prove the weaker assertions \eqref{eq:heur} and \eqref{eq:ekmek}. Then, we will combine \eqref{eq:heur} (with $\theta = \theta_0 \pm c$) and \eqref{eq:ekmek} to deduce the existence of $\theta_1^*,\theta_2^*\in\mathbb{R}$ such that
\[ \theta_0-c < \theta_1^* < \theta_0 < \theta_2^* < \theta_0+c, \]
\[ \ol{H}(\theta_0-c) < \ol{H}(\theta_1^*)\quad\text{and}\quad\ol{H}(\theta_2^*) > \ol{H}(\theta_0+c), \]
and thereby conclude that $\ol{H}$ is not quasiconvex on $[\theta_0 - c, \theta_0 + c]$. This will give Theorem \ref{thm:hicaz} under assumption \eqref{eq:resmi}. Examining the details of the construction and the argument, we will prove Theorem \ref{thm:hicaz} under the more general assumption \eqref{eq:Gduz}.

\subsection{Higher dimensions}\label{ss:higher}

To extend our counterexamples to any dimension $d\ge 2$, we observe that, if $v_1$ solves
\begin{align}
	\pd_t v&=\pd^2_{xx} v+G_1(\pd_{x} v)+V_1(x),\quad (t,x)\in(0,\infty)\times\mathbb{R},\label{eq:md1}\\
	v&(0,x)=g_1(x),\quad x\in\mathbb{R},\nonumber
	\shortintertext
	{and $\breve{v}_i$ solves}
	\pd_t \breve{v}&=\pd^2_{xx} \breve{v}+\breve{G}(\pd_x \breve{v})+\breve{V}(x),\quad (t,x)\in(0,\infty)\times\mathbb{R},\label{eq:md2}\\
	\breve{v}&(0,x)=g_i(x),\quad x\in\mathbb{R},\nonumber
\end{align}
for each $i=2,3,\ldots,d$, then $u(t,x) := v_1(t,x_1)+\sum_{i=2}^d\breve{v}_i(t,x_i)$ solves 
\begin{align}
	\pd_tu& = \Delta_xu + G_1(\pd_{x_1}u) + \sum_{i=2}^d\breve{G}(\pd_{x_i}u) + V_1(x_1) + \sum_{i=2}^d \breve{V}(x_i),\quad(t,x)\in(0,\infty)\times\mathbb{R}^d,\label{eq:mdeq}\\
	u&(0,x) = \sum_{i=1}^dg_i(x_i),\quad x\in\mathbb{R}^d.\label{eq:mdinit}
\end{align}

We assume that the potentials $V_1,\breve{V}\in \Lip(\mathbb{R})$ are $1$-periodic, $G_1$ is quasiconvex, $\breve{G}$ is convex, $G_1,\breve{G}\in \Con^2(\mathbb{R})\cap\mathcal{G}_1$, and
all $g_i\in\UC(\mathbb{R})$. Under these assumptions, the solution $u$ of the problem \eqref{eq:mdeq}--\eqref{eq:mdinit} is unique.

Equations \eqref{eq:md1} and \eqref{eq:md2} homogenize. If we denote the corresponding effective Hamiltonians by $\ol{H}_1$ and $\ol{\mathcal{H}}$, respectively, then, for each $\theta\in\mathbb{R}$, there are $1$-periodic $f_{\theta},\breve{f}_{\theta}\in\Con^1(\mathbb{R})$ as in \eqref{eq:atar} such that
\begin{equation}\label{eq:mado}
	\int_0^1f_\theta(x)dx = \int_0^1\breve{f}_\theta(x)dx = \theta,
\end{equation}
\begin{equation}\label{eq:ams} 
	f_\theta'(x) + G_1(f_\theta(x)) + V_1(x) = \ol{H}_1(\theta)\quad\text{and}\quad \breve{f}_\theta'(x) + \breve{G}(\breve{f}_\theta(x)) + \breve{V}(x) = \ol{{\mathcal H}}(\theta)
\end{equation}
for all $x\in\mathbb{R}$. It follows that
\begin{equation}
	\label{eq:vees}
	v_{\theta}(t,x) := t\ol{H}_1(\theta)+\int_0^xf_{\theta}(y)\,dy\ \ \text{and}\ \  \breve{v}_{\theta}(t,x) := t\ol{{\mathcal H}}(\theta)+\int_0^x\breve{f}_{\theta}(y)\,dy
\end{equation}
are solutions of \eqref{eq:md1} and \eqref{eq:md2}, respectively, \eqref{eq:mdeq} homogenizes, and the corresponding effective Hamiltonian is given by
\begin{equation}\label{eq:effHsum}
	\ol{H}_1(\theta_1)+\sum_{i=2}^d\ol{\mathcal H}(\theta_i),\quad \theta\in\mathbb{R}^d.
\end{equation}

Note that, if $\ol{H}_1$ is not quasiconvex (on $\mathbb{R}$), then the sum in \eqref{eq:effHsum} is not quasiconvex (on $\mathbb{R}^d$). Therefore, to prove Theorem \ref{thm:multi}, it would
be sufficient to take a quasiconvex $G_1$ and a $1$-periodic $V_1\in \Lip(\mathbb{R})$ as in Theorem \ref{thm:hicaz} such that the corresponding $\ol{H}_1$ is not quasiconvex, and construct a convex $\breve{G}$ for which
\[ G(p):=G_1(p_1)+\sum_{i=2}^d\breve{G}(p_i) \]
is quasiconvex (on $\mathbb{R}^d$). However, the last step cannot be achieved on the whole $\mathbb{R}^d$. 

In Section \ref{sec:multi}, we will carry out the construction of $G$ as above on a bounded convex set $S$ such that all of the ``action'' needed for our result happens inside $S$. Then, we will extend $G$ from $S$ to $\mathbb{R}^d$ in a quasiconvex way (without preserving the separation of variables outside of $S$), and thereby establish Theorem \ref{thm:multi}.

\section{Analysis of correctors}\label{sec:general}

In this section, we assume that $G\in\Con^1(\mathbb{R})$ is coercive and $V\in\Lip(\mathbb{R})$ is $1$-periodic. Then, for each $\theta\in\mathbb{R}$, there exists a unique $\ol{H}(\theta)\in\mathbb{R}$ and a $1$-periodic $F_\theta\in\Con^2(\mathbb{R})$, referred to as the {\em corrector}, such that \eqref{eq:biz} holds for all $x\in\mathbb{R}$.
This follows from \cite[Lemma 4.1]{Evans92} and some additional arguments, see Proposition \ref{prop:tax} in Appendix \ref{app:1d}. We start by recording this result in an equivalent way.

\begin{lemma}\label{lem:equi}
	For each $\theta\in\mathbb{R}$, there exists a unique $\ol{H}(\theta)\in\mathbb{R}$ and a unique $1$-periodic $f_\theta\in\Con^1(\mathbb{R})$ such that
\begin{equation}\label{eq:sifir}
	\int_0^1f_\theta(x)dx = \theta,
\end{equation}
and
\begin{equation}\label{eq:ODE}
	f_\theta'(x) + G(f_\theta(x)) + V(x) = \ol{H}(\theta)\quad\text{for all}\ x\in\mathbb{R}.
\end{equation}
\end{lemma}

\begin{proof}
	Define $f_\theta\in\Con^1(\mathbb{R})$ by setting $f_\theta(x) = \theta + F_\theta'(x)$ for all $x\in\mathbb{R}$. Note that $f_\theta$ is $1$-periodic and it satisfies \eqref{eq:sifir}--\eqref{eq:ODE}.
	
	It remains to prove the uniqueness of $\ol{H}(\theta)$ and $f_\theta$. Take any $\ol{H}(\theta)\in\mathbb{R}$ and $1$-periodic $f_\theta\in\Con^1(\mathbb{R})$ such that \eqref{eq:sifir}--\eqref{eq:ODE} hold. Define $F_\theta\in\Con^2(\mathbb{R})$ by setting $F_\theta(x) = \int_0^xf_\theta(y)dy - \theta x$ for all $x\in\mathbb{R}$. Note that $F_\theta$ is $1$-periodic and it satisfies \eqref{eq:biz}. Therefore, $\ol{H}(\theta)$ is unique by Proposition \ref{prop:tax}. Finally, if there is another $1$-periodic $f\in\Con^1(\mathbb{R})$ (distinct from $f_\theta$) that solves \eqref{eq:ODE}, then $f$ and $f_\theta$ are strictly ordered. Therefore, \eqref{eq:sifir} cannot hold for $f$.
\end{proof}

Consider the equation
\begin{equation}\label{eq:derODE}
	g_\theta'(x) + G'(f_\theta(x))g_\theta(x) = c(\theta),\quad x\in\mathbb{R},
\end{equation}
with any $c(\theta)\in\mathbb{R}$.\footnote{\label{foot:circum} Equation \eqref{eq:derODE} is obtained by formally differentiating both sides of \eqref{eq:ODE} with respect to $\theta$, but we will circumvent the question as to whether $\ol{H}'(\theta)$ and $\pd_\theta f_\theta$ exist, see Section \ref{ss:heur1d}.}
The general solution of this equation is given by
\begin{align}
	g_\theta(x) &= \left(c(\theta)B_\theta(x) + C(\theta)\right)e^{-I_\theta(x)},\label{eq:solucan}\\
	\shortintertext{where}
	I_\theta(x) &= \int_0^x G'(f_\theta(y))dy,\quad B_\theta(x) = \int_0^xe^{I_\theta(y)}dy,\label{eq:intfac}
\end{align}
and $C(\theta) \in \mathbb{R}$ is an arbitrary constant.

We are interested in $1$-periodic positive solutions $g_\theta\in\Con^1(\mathbb{R})$ to equation \eqref{eq:derODE}. This will limit the set of values $c(\theta)$ can take.
Since $f_\theta$ is $1$-periodic, we see from \eqref{eq:derODE} that $g_\theta$ is $1$-periodic if and only if $g_\theta(1) = g_\theta(0)$, i.e.,
\[ \left(c(\theta)B_\theta(1) + C(\theta)\right)e^{-I_\theta(1)} = C(\theta) \]
by \eqref{eq:solucan}--\eqref{eq:intfac}, which can be rearranged as
\begin{equation}\label{eq:sart}
	C(\theta)(e^{I_\theta(1)} - 1) = c(\theta)B_\theta(1).
\end{equation}
Note that $C(\theta) = g_\theta(0) > 0$ (since we want $g_\theta$ to be positive) and $B_\theta(1) = \int_0^1e^{I_\theta(x)}dx > 0$.

From the identity in \eqref{eq:sart}, we infer the set of values $c(\theta)$ can take. There are three cases:
\begin{itemize}
	\item [(i)] If $I_\theta(1) > 0$, then $c(\theta) > 0$ and $C(\theta) = \displaystyle{\frac{c(\theta)B_\theta(1)}{e^{I_\theta(1)} - 1}} > 0$.
	\item [(ii)] If $I_\theta(1) < 0$, then $c(\theta) < 0$ and $C(\theta) = \displaystyle{\frac{c(\theta)B_\theta(1)}{e^{I_\theta(1)} - 1}} > 0$.
	\item [(iii)] If $I_\theta(1) = 0$, then $c(\theta) = 0$ and $C(\theta) > 0$ is arbitrary (which is clear because taking $c(\theta) = 0$ turns \eqref{eq:derODE} into an homogeneous equation). %We will not use this case in our analysis.
\end{itemize}
In each of these three cases, we have found a $1$-periodic positive $g_\theta\in\Con^1(\mathbb{R})$ that solves \eqref{eq:derODE}. Let
\begin{equation}\label{eq:pb}
	b(\theta) = \int_0^1 g_\theta(x) dx > 0.
\end{equation}

For any $h\in\mathbb{R}$, define
\[ f_\theta^h(x) = f_\theta(x) + hg_\theta(x). \]
Note that $f_\theta^h \in\Con^1(\mathbb{R})$, it is $1$-periodic, and
\begin{equation}\label{eq:peas}
	\begin{split}
		(f_\theta^h)'(x) + G(f_\theta^h(x)) + V(x) &= (f_\theta(x) + hg_\theta(x))' + G(f_\theta(x) + h g_\theta(x)) + V(x)\\
		&= f_\theta'(x) + hg_\theta'(x) + G(f_\theta(x)) + hG'(f_\theta(x))g_\theta(x) + o(h) + V(x)\\
		&= f_\theta'(x) + G(f_\theta(x)) + V(x) + h[ g_\theta'(x) + G'(f_\theta(x))g_\theta(x) ] + o(h)\\
		&= \ol{H}(\theta) + c(\theta)h + o(h)\quad \text{for all}\ x\in\mathbb{R},
	\end{split}
\end{equation}
by \eqref{eq:ODE}--\eqref{eq:derODE}. We will now use \eqref{eq:peas} to deduce information about $\ol{H}$ near $\theta$.

\smallskip

\textbf{Case 1: $I_\theta(1) > 0$.} Take $c(\theta) = 1$. For every $a\in(0,1)$, we see from \eqref{eq:peas} that
$f_\theta^h(x)$ is a strict supersolution\footnote{In this first-order ODE context, we say that an $f\in\Con^1(\mathbb{R})$ is a strict subsolution (resp.\ supersolution) of \eqref{eq:kader} if it satisfies \eqref{eq:kader} when the ``$=$" sign there is replaced by the ``$<$" (resp. ``$>$") sign.} to
\begin{equation}\label{eq:kader}
	f'(x) + G(f(x)) + V(x) = \ol{H}(\theta) + ah,\quad x\in\mathbb{R},
\end{equation}
when $h>0$ is sufficiently small. Since $f_\theta(x)$ is a strict subsolution to \eqref{eq:kader} when $h>0$, the set
\[ \mathcal{S}_\theta^{a,h} = \{f\in\Con^1(\mathbb{R}):\,\text{$f$ solves \eqref{eq:kader} and satisfies $f_\theta(x) < f(x) < f_\theta^h(x)$ for all $x\in\mathbb{R}$}\} \]
is nonempty by \cite[Lemma A.8]{DKY23+}, and it is a compact subset of $\Con(\mathbb{R})$ under the topology of locally uniform convergence. Let $f_\theta^{a,h}(x) = \sup\{f(x):\,f\in\mathcal{S}_\theta^{a,h}\}$. Then, $f_\theta^{a,h}\in\mathcal{S}_\theta^{a,h}$ by \cite[Lemma A.9]{DKY23+}. Moreover, since $f_\theta$, $f_\theta^h$ and $V$ are $1$-periodic, we know that, for every $f\in\Con^1(\mathbb{R})$, $f\in\mathcal{S}_\theta^{a,h}$ if and only if $f(\,\cdot\,\pm 1)\in\mathcal{S}_\theta^{a,h}$. Therefore,
\[ f_\theta^{a,h}(x+1) = \sup\{f(x+1):\,f\in\mathcal{S}_\theta^{a,h}\} = \sup\{f(x):\,f\in\mathcal{S}_\theta^{a,h}\} = f_\theta^{a,h}(x) \]
for all $x\in\mathbb{R}$, i.e., $f_\theta^{a,h}$ is $1$-periodic. Let
\[ \theta^*(\theta,a,h) = \int_0^1f_\theta^{a,h}(x)dx. \]
Then,
\[ \ol{H}(\theta^*(\theta,a,h)) = \ol{H}(\theta) + ah \]
by Lemma \ref{lem:equi}. Finally, note that
\[ \theta < \theta^*(\theta,a,h) < \theta + hb(\theta) \]
by \eqref{eq:sifir} and \eqref{eq:pb}.

\smallskip

\textbf{Case 2: $I_\theta(1) < 0$.} Take $c(\theta) = -1$. For every $a\in(-1,0)$, we see from \eqref{eq:peas} that $f_\theta^h(x)$ is a strict supersolution to \eqref{eq:kader} when $h<0$ and $|h|$ is sufficiently small. Since $f_\theta(x)$ is a strict subsolution to \eqref{eq:kader} when $h<0$, by an argument analogous to the one we gave in Case 1, there exists a $1$-periodic $f_\theta^{a,h}\in\Con^1(\mathbb{R})$ that solves \eqref{eq:kader} and satisfies $f_\theta^h(x) < f_\theta^{a,h}(x) < f_\theta(x)$ for all $x\in\mathbb{R}$. Let
\[ \theta^*(\theta,a,h) = \int_0^1f_\theta^{a,h}(x)dx \]
as in Case 1. Then,
\[ \ol{H}(\theta^*(\theta,a,h)) = \ol{H}(\theta) + ah \]
by Lemma \ref{lem:equi}. Finally, note that
\[ \theta + hb(\theta) < \theta^*(\theta,a,h) < \theta \]
by \eqref{eq:sifir} and \eqref{eq:pb}.

\smallskip

We recapitulate these results below with the choices $a = 1/2$ and $a = -1/2$ in Cases 1 and 2, respectively.

\begin{lemma}\label{lem:yous}
	Let $I_\theta(1) = \int_0^1 G'(f_\theta(x))dx$.
	\begin{itemize}
		\item [(a)] If $I_\theta(1) > 0$, then, for sufficiently small $h>0$, there is a $\theta^*\in(\theta,\theta + hb(\theta))$ such that
		\[ \ol{H}(\theta^*) = \ol{H}(\theta) + h/2 > \ol{H}(\theta). \]
		\item [(b)] If $I_\theta(1) < 0$, then, for sufficiently small $h>0$, there is a $\theta^*\in(\theta - hb(\theta),\theta)$ such that
		\[ \ol{H}(\theta^*) = \ol{H}(\theta) + h/2 > \ol{H}(\theta). \]
	\end{itemize}
\end{lemma}

Next, we make an elementary observation.

\begin{lemma}\label{lem:order}
	For every $\theta_1,\theta_2\in\mathbb{R}$, if $\theta_1 < \theta_2$, then $f_{\theta_1}(x) < f_{\theta_2}(x)$ for all $x\in\mathbb{R}$.
\end{lemma}

\begin{proof}
	See Appendix \ref{app:lemmas}.
\end{proof}

For every $\theta_1,\theta_2\in\mathbb{R}$ such that $\theta_1 < \theta_2$, let $g_{\theta_1,\theta_2} = f_{\theta_2} - f_{\theta_1}$. Note that $g_{\theta_1,\theta_2}\in\Con^1(\mathbb{R})$ and it is $1$-periodic. Moreover, $g_{\theta_1,\theta_2}(x) > 0$ for all $x\in\mathbb{R}$ by Lemma \ref{lem:order}. In the final result of this section, we provide lower and upper bounds for $(\theta_2 - \theta_1)^{-1}g_{\theta_1,\theta_2}$.\footnote{Since we circumvent the question as to whether $\pd_\theta f_\theta$ exists (see Section \ref{ss:heur1d} and Footnote \ref{foot:circum}), we instead work with the difference quotient $\frac{f_{\theta_2} - f_{\theta_1}}{\theta_2 - \theta_1}$.}

\begin{lemma}\label{lem:bounds}
	For every $\theta_1,\theta_2\in\mathbb{R}$, if $\theta_1 < \theta_2$, then
	\begin{equation}\label{eq:bounds}
		(\theta_2 - \theta_1)e^{-K_1(\theta_1,\theta_2)} \le g_{\theta_1,\theta_2}(x) \le (\theta_2 - \theta_1)e^{K_1(\theta_1,\theta_2)}
	\end{equation}
	for all $x\in\mathbb{R}$, where
	\[ K_1(\theta_1,\theta_2) = \max\left\{ |G'(p)|:\,\min_{x\in[0,1]}f_{\theta_1}(x) \le p \le \max_{x\in[0,1]}f_{\theta_2}(x) \right\}. \]	
\end{lemma}

\begin{proof}
	See Appendix \ref{app:lemmas}.
\end{proof}

\section{The construction in one dimension}\label{sec:construct}

Take any coercive $G\in\Con^2(\mathbb{R})$ and $p_1,p_2\in\mathbb{R}$ such that $p_1 < p_2$ and $G'(p_1) < 0 < G'(p_2)$. Let
\[ m_1 = \min\{G'(p):\,p\in[p_1,p_2]\},\quad M_1 = \max\{G'(p):\,p\in[p_1,p_2]\}, \]
and
\begin{equation}\label{eq:liman}
	L = \frac{G'(p_2)}{G'(p_2) - G'(p_1)} \in (0,1).
\end{equation}
Note that $G'(p_1)L + G'(p_2)(1-L) = 0$, and for any $\ell \in (0,L\wedge(1-L))$ (to be chosen later),
\begin{equation}\label{eq:amb}
\begin{aligned}
	&G'(p_1)(L-\ell) + G'(p_2)(1-L-\ell) + 2\ell m_1 < 0\\
	<\, &G'(p_1)(L-\ell) + G'(p_2)(1-L-\ell) + 2\ell M_1.
\end{aligned}
\end{equation}
Define
\begin{align*}
	\mathcal{C}(L,\ell) &= \left\{f\in\Con^{1,1}(\mathbb{R}):\, \text{$f$ is $1$-periodic},\ p_1 \le f(x) \le p_2\ \text{for all}\ x\in\mathbb{R},\right.\\
													  &\hspace{29mm} \left. f(x) = p_1\ \text{for}\ x\in[0,L - \ell],\ \text{and}\ f(x) = p_2\ \text{for}\ x\in[L,1-\ell]\right\},
\end{align*}
where $\Con^{1,1}(\mathbb{R})$ denotes the set of all $f\in\Con^1(\mathbb{R})$ such that $f'\in\Lip(\mathbb{R})$.

\begin{lemma}\label{lem:var}
	For any $\ell \in (0,L\wedge(1-L))$, there is an $f\in \mathcal{C}(L,\ell)$ such that $\int_0^1 G'(f(x))dx = 0$.
\end{lemma}

\begin{proof}
	Since
	\begin{align*}
		&G'(p_1)(L-\ell) + G'(p_2)(1-L-\ell) + 2\ell m_1 \le \int_0^1 G'(f(x))dx\\
		\le\, &G'(p_1)(L-\ell) + G'(p_2)(1-L-\ell) + 2\ell M_1
	\end{align*}
	for every $f\in \mathcal{C}(L,\ell)$ and we have \eqref{eq:amb}, it is easy to see that $\int_0^1 G'(f(x))dx = 0$ for some $f\in \mathcal{C}(L,\ell)$. We provide an example for the sake of completeness. There exist $p_{\text{min}},p_{\text{max}}\in[p_1,p_2]$ such that $G'(p_{\text{min}}) = m_1$ and $G'(p_{\text{max}}) = M_1$. Pick $\ell'>0$ small enough so that
	\begin{equation}\label{eq:siq}
		\begin{aligned}
			&G'(p_1)(L-\ell) + G'(p_2)(1-L-\ell) + 2(\ell - \ell')m_1 + 2\ell'M_1 < 0\\
			<\, &G'(p_1)(L-\ell) + G'(p_2)(1-L-\ell) + 2(\ell - \ell') M_1 + 2\ell'm_1.
		\end{aligned}
	\end{equation}
	Take an $f\in\mathcal{C}(L,\ell)$ that spends $2(\ell - \ell')(1-a)$ amount of ``time" at $p_{\text{min}}$, $2(\ell - \ell')a$ amount of ``time" at $p_{\text{max}}$, and the remaining $2\ell'$ amount of ``time" elsewhere in $[p_1,p_2]$. Here, $a\in(0,1)$ is a parameter that we will tune, and we will not change $f$ in any other way. Observe that
	\begin{align*}
		&G'(p_1)(L-\ell) + G'(p_2)(1-L-\ell) + 2(\ell - \ell')(1-a)m_1 + 2(\ell - \ell')aM_1 + 2\ell'm_1\\
		\le\, &\int_0^1 G'(f(x))dx\\
		\le\, &G'(p_1)(L-\ell) + G'(p_2)(1-L-\ell) + 2(\ell - \ell')(1-a)m_1 + 2(\ell - \ell')aM_1 + 2\ell'M_1.
	\end{align*}
	We see from \eqref{eq:siq} that $\int_0^1 G'(f(x))dx$ is negative (resp.\ positive) for $a=0$ (resp.\ $a=1$). By the intermediate value theorem, $\int_0^1 G'(f(x))dx = 0$ for some $a\in(0,1)$.
\end{proof}

Let $f\in \mathcal{C}(L,\ell)$ be as in Lemma \ref{lem:var}. Define a potential $V:\mathbb{R}\to\mathbb{R}$ by
\begin{equation}\label{eq:rox}
	V(x) = - f'(x) - G(f(x)),
\end{equation}
so that
\begin{equation}\label{eq:hafta}
	f'(x) + G(f(x)) + V(x) = 0\quad\text{for all}\ x\in\mathbb{R}.
\end{equation}
Note that $V\in\Lip(\mathbb{R})$ and it is $1$-periodic, hence our results in Section \ref{sec:general} are applicable.

Let \[ \theta_0 = \int_0^1 f(x)dx. \]
Then, $\ol{H}(\theta_0) = 0$ and $f_{\theta_0} = f$ by \eqref{eq:hafta} and Lemma \ref{lem:equi}. Moreover, recalling \eqref{eq:intfac}, we have
\begin{equation}\label{eq:isot}
	I_{\theta_0}(1) = \int_0^1 G'(f(x))dx = 0
\end{equation}
since we have taken $f$ as in Lemma \ref{lem:var}.

Recall the notation in Lemma \ref{lem:bounds}, and let
\[ K_1 = \max\{|G'(p)|:\,p\in[p_1,p_2]\} = \max\{-m_1,M_1\}. \]

\begin{lemma}\label{lem:isim}
	$K_1(\theta_0,\theta_0 + c) \downarrow K_1$ and $K_1(\theta_0 - c,\theta_0) \downarrow K_1$ as $c\downarrow 0$.
\end{lemma}

\begin{proof}
	Note that $\displaystyle{\min_{x\in[0,1]}f_{\theta_0}(x) = p_1}$ and $\displaystyle{\max_{x\in[0,1]}f_{\theta_0}(x) = p_2}$ by construction.
	
	Since the mapping $c\mapsto f_{\theta_0 + c}(x)$ is increasing for all $x\in[0,1]$ by Lemma \ref{lem:order}, it is clear that $c\mapsto K_1(\theta_0,\theta_0 + c)$ is nondecreasing. Thus, $f_{\theta_0 + c}(x) - f_{\theta_0}(x) = g_{\theta_0,\theta_0 + c}(x) \le ce^{K_1(\theta_0,\theta_0 + 1)}$ for all $c\in(0,1]$ by \eqref{eq:bounds}. In particular, $\displaystyle{\max_{x\in[0,1]}f_{\theta_0 + c}(x)\downarrow p_2}$, and hence, $K_1(\theta_0,\theta_0 + c) \downarrow K_1$ as $c\downarrow 0$.
	
	Similarly, since $c\mapsto f_{\theta_0 - c}(x)$ is decreasing for all $x\in[0,1]$ by Lemma \ref{lem:order}, it is clear that $c\mapsto K_1(\theta_0 - c,\theta_0)$ is nondecreasing. Thus, $f_{\theta_0}(x) - f_{\theta_0 - c}(x) = g_{\theta_0 - c,\theta_0}(x) \le ce^{K_1(\theta_0 - 1,\theta_0)}$ for all $c\in(0,1]$ by \eqref{eq:bounds}. In particular, $\displaystyle{\min_{x\in[0,1]}f_{\theta_0 - c}(x)\uparrow p_1}$, and hence, $K_1(\theta_0 - c,\theta_0) \downarrow K_1$ as $c\downarrow 0$.
\end{proof}

For every $\delta \ge 0$, let
\[ K_2(\delta) = \max\{|G''(p)|:\,p\in[p_1-\delta,p_2+\delta]\}, \]
and write $K_2 = K_2(0)$ for the sake of notational brevity. Denote by $\rho:[0,\infty)\to[0,\infty)$ the modulus of continuity of $G''$ on $[p_1,p_2]$.
	
\begin{lemma}\label{lem:yekun}
	Assume that
	\begin{equation}\label{eq:cond1}
		G''(p_1) < 0\quad\text{and}\quad G''(p_2) < 0.
	\end{equation}
	Recall \eqref{eq:liman} and pick $\ell \in (0,L\wedge(1-L))$ small enough so that
	\begin{equation}\label{eq:pickup1}
		\left[(L-\ell)G''(p_1) + (1-L-\ell)G''(p_2)\right]e^{-K_1} + 2\ell K_2e^{K_1} < 0.
	\end{equation}
	Then, $I_{\theta_0 + c}(1) < 0 < I_{\theta_0 - c}(1)$ when $c > 0$ is sufficiently small.
\end{lemma}

\begin{proof}
	By \eqref{eq:cond1}--\eqref{eq:pickup1}, we can fix a $\delta\in(0,1)$ small enough so that $G''(p_i) + \rho(\delta) < 0$ for $i=1,2$, and
	\begin{equation}\label{eq:said1}
		\left[(L-\ell)(G''(p_1) + \rho(\delta)) + (1-L-\ell)(G''(p_2) + \rho(\delta))\right]e^{-K_1} + 2\ell K_2(\delta)e^{K_1} < 0.
	\end{equation}
	
	Recall from \eqref{eq:isot} that $I_{\theta_0}(1) = 0$. For any $c > 0$,
	\begin{equation}\label{eq:diz11}
		\begin{aligned}
			I_{\theta_0 + c}(1) = I_{\theta_0 + c}(1) - I_{\theta_0}(1) &= \int_0^1\left(G'(f_{\theta_0 + c}(x)) - G'(f_{\theta_0}(x))\right)dx\\
			&= \int_0^1 G''(f^*(x)) g_{\theta_0,\theta_0+c}(x)dx
		\end{aligned}
	\end{equation}
	for some $f^*(x) \in \left( f_{\theta_0}(x), f_{\theta_0 + c}(x) \right)$ by the mean value theorem. By \eqref{eq:bounds},
	\begin{equation}\label{eq:alv1}
		f^*(x) - f_{\theta_0}(x) \le g_{\theta_0,\theta_0+c}(x) \le ce^{K_1(\theta_0,\theta_0+c)}
	\end{equation}
	for all $x\in\mathbb{R}$. Pick $c>0$ small enough so that $ce^{K_1(\theta_0,\theta_0+c)} \le \delta$. Then,
	\begin{equation}\label{eq:sol11}
		\begin{aligned}
			&G''(f^*(x)) \le K_2(\delta)\quad\text{for all $x\in\mathbb{R}$, and}\\
			&G''(f^*(x)) \le G''(p_i) + \rho(\delta)\quad\text{when $f_{\theta_0}(x) = p_i$ for $i=1,2$.}
		\end{aligned}
	\end{equation}
	Going back to \eqref{eq:diz11} and using the bounds in \eqref{eq:bounds} and \eqref{eq:sol11}, we obtain the following inequality:
	\begin{align*}
		I_{\theta_0 + c}(1) &\le (L-\ell)(G''(p_1) + \rho(\delta))ce^{-K_1(\theta_0,\theta_0+c)} + (1-L-\ell)(G''(p_2) + \rho(\delta))ce^{-K_1(\theta_0,\theta_0+c)}\\
		&\quad + 2\ell K_2(\delta)ce^{K_1(\theta_0,\theta_0+c)}.
	\end{align*}		
	Rearranging the right-hand side of this inequality, we see that
	\begin{align*}
		\frac1{c}\,I_{\theta_0 + c}(1) &\le \left[(L-\ell)(G''(p_1) + \rho(\delta)) + (1-L-\ell)(G''(p_2) + \rho(\delta))\right]e^{-K_1(\theta_0,\theta_0+c)}\\
		&\quad + 2\ell K_2(\delta)e^{K_1(\theta_0,\theta_0+c)}.
	\end{align*}
	Recalling \eqref{eq:said1} and Lemma \ref{lem:isim}, we conclude that $I_{\theta_0 + c}(1) < 0$ for $c>0$ sufficiently small.
	
	A similar argument (with $\theta_0 - c$ and $K_1(\theta_0-c,\theta_0)$ in place of $\theta_0 + c$ and $K_1(\theta_0,\theta_0+c)$, respectively) shows that $-I_{\theta_0 - c}(1) < 0$ for $c>0$ sufficiently small.
\end{proof}

\begin{lemma}\label{lem:yuled}
	Assume that
	\begin{equation}\label{eq:cond2}
		G''(p_1) < 0,\quad G''(p_2) \ge 0\quad\text{and}\quad G''(p_1)G'(p_2)e^{-K_1} < G''(p_2)G'(p_1)e^{K_1}.
	\end{equation}
	Note that the last inequality is equivalent to
	\[ LG''(p_1)e^{-K_1} + (1-L)G''(p_2)e^{K_1} < 0 \]
	with $L\in(0,1)$ defined in \eqref{eq:liman}. Pick $\ell \in (0,L\wedge(1-L))$ small enough so that
	\begin{equation}\label{eq:pickup2}
		(L-\ell)G''(p_1)e^{-K_1} + \left[(1-L-\ell)G''(p_2) + 2\ell K_2\right]e^{K_1} < 0.
	\end{equation}
	Then, $I_{\theta_0 + c}(1) < 0 < I_{\theta_0 - c}(1)$ when $c > 0$ is sufficiently small.
\end{lemma}

\begin{proof}
	By \eqref{eq:cond2}--\eqref{eq:pickup2}, we can fix a $\delta\in(0,1)$ small enough so that $G''(p_1) + \rho(\delta) < 0$ and
	\begin{equation}\label{eq:said2}
		(L-\ell)(G''(p_1) + \rho(\delta))e^{-K_1} + \left[(1-L-\ell)(G''(p_2) + \rho(\delta)) + 2\ell K_2(\delta)\right]e^{K_1} < 0.
	\end{equation}
	
	We follow the proof of Lemma \ref{lem:yekun}. Recall from \eqref{eq:isot} that $I_{\theta_0}(1) = 0$. For any $c > 0$, \eqref{eq:diz11} holds for some $f^*(x) \in \left( f_{\theta_0}(x), f_{\theta_0 + c}(x) \right)$ by the mean value theorem. By \eqref{eq:bounds}, the inequalities in \eqref{eq:alv1} are valid
	%\[ f^*(x) - f_{\theta_0}(x) \le g_{\theta_0,\theta_0+c}(x) \le ce^{K_1(\theta_0,\theta_0+c)} \]
	for all $x\in\mathbb{R}$. Pick $c>0$ small enough so that $ce^{K_1(\theta_0,\theta_0+c)} \le \delta$. Then, we have \eqref{eq:sol11}. Going back to \eqref{eq:diz11} and using the bounds in \eqref{eq:bounds} and \eqref{eq:sol11}, we obtain the following inequality:
	\begin{align*}
		I_{\theta_0 + c}(1) &\le (L-\ell)(G''(p_1) + \rho(\delta))ce^{-K_1(\theta_0,\theta_0+c)} + (1-L-\ell)(G''(p_2) + \rho(\delta))ce^{K_1(\theta_0,\theta_0+c)}\\
		&\quad + 2\ell K_2(\delta)ce^{K_1(\theta_0,\theta_0+c)}.
	\end{align*}		
	Rearranging the right-hand side of this inequality, we see that
	\begin{align*}
		\frac1{c}\,I_{\theta_0 + c}(1) &\le (L-\ell)(G''(p_1) + \rho(\delta))e^{-K_1(\theta_0,\theta_0+c)}\\
		&\quad + \left[(1-L-\ell)(G''(p_2) + \rho(\delta)) + 2\ell K_2(\delta)\right]e^{K_1(\theta_0,\theta_0+c)}.
	\end{align*}
	Recalling \eqref{eq:said2} and Lemma \ref{lem:isim}, we conclude that $I_{\theta_0 + c}(1) < 0$ for $c>0$ sufficiently small.
	
	A similar argument (with $\theta_0 - c$ and $K_1(\theta_0-c,\theta_0)$ in place of $\theta_0 + c$ and $K_1(\theta_0,\theta_0+c)$, respectively) shows that $-I_{\theta_0 - c}(1) < 0$ for $c>0$ sufficiently small.
	\begin{comment}
	Similarly, for any $c > 0$, \eqref{eq:diz12} holds for some $f_*(x) \in \left( f_{\theta_0 - c}(x), f_{\theta_0}(x) \right)$ by the mean value theorem. By \eqref{eq:bounds}, the inequalities in \eqref{eq:alv2} are valid
	%\[ f_{\theta_0}(x) - f_*(x) \le g_{\theta_0-c,\theta_0}(x) \le ce^{K_1(\theta_0-c,\theta_0)} \]
	for all $x\in\mathbb{R}$. Pick $c>0$ small enough so that $ce^{K_1(\theta_0-c,\theta_0)} \le \delta$. Then, we have \eqref{eq:sol12}. Going back to \eqref{eq:diz12} and using the bounds in \eqref{eq:bounds} and \eqref{eq:sol12}, we obtain the following inequality:
	\begin{align*}
		-I_{\theta_0 - c}(1) &\le (L-\ell)(G''(p_1) + \rho(\delta))ce^{-K_1(\theta_0-c,\theta_0)} + (1-L-\ell)(G''(p_2) + \rho(\delta))ce^{K_1(\theta_0-c,\theta_0)}\\
		&\quad + 2\ell K_2(\delta)ce^{K_1(\theta_0-c,\theta_0)}.
	\end{align*}		
	Rearranging the right-hand side of this inequality, we see that
	\begin{align*}
		-\frac1{c}\,I_{\theta_0 - c}(1) &\le (L-\ell)(G''(p_1) + \rho(\delta))e^{-K_1(\theta_0-c,\theta_0)}\\
		&\quad + \left[(1-L-\ell)(G''(p_2) + \rho(\delta)) + 2\ell K_2(\delta)\right]e^{K_1(\theta_0-c,\theta_0)}.
	\end{align*}
	Recalling \eqref{eq:said2} and Lemma \ref{lem:isim}, we conclude that $-I_{\theta_0 - c}(1) < 0$ for $c>0$ sufficiently small.
	\end{comment}
\end{proof}

%Putting everything together, we establish our main result in one dimension.

\begin{proof}[Proof of Theorem \ref{thm:hicaz}]
	Recall \eqref{eq:liman}. Under the assumption \eqref{eq:Gduz}, either \eqref{eq:cond1} or \eqref{eq:cond2} is true. In the former (resp.\ latter) case, 
	pick $\ell \in (0,L\wedge(1-L))$ small enough so that \eqref{eq:pickup1} (resp.\ \eqref{eq:pickup2}) holds. Let $f\in \mathcal{C}(L,\ell)$ be as in Lemma \ref{lem:var}, and define $V\in\Lip(\mathbb{R})$ by \eqref{eq:rox}. Fix a sufficiently small $c>0$ so that $I_{\theta_0 - c}(1) > 0 > I_{\theta_0 + c}(1)$ by Lemma \ref{lem:yekun} (resp.\ Lemma \ref{lem:yuled}) in the former (resp.\ latter) case. Apply Lemma \ref{lem:yous} (with $\theta = \theta_0 \pm c$) to deduce that, for sufficiently small $h>0$, there exist $\theta_1^*,\theta_2^*\in\mathbb{R}$ such that
	\[ \theta_0-c < \theta_1^* < \theta_0-c + hb(\theta_0-c) < \theta_0 < \theta_0+c - hb(\theta_0+c) < \theta_2^* < \theta_0+c, \]
	\[ \ol{H}(\theta_0-c) < \ol{H}(\theta_1^*)\quad\text{and}\quad\ol{H}(\theta_2^*) > \ol{H}(\theta_0+c). \]
	We conclude that $\ol{H}$ is not quasiconvex on $[\theta_0 - c, \theta_0 + c]$.
\end{proof}

\begin{proof}[Proof of Corollary \ref{cor:modify}]
	We will work with the bump function $\Psi\in\Con^2(\mathbb{R})$ defined by
	\[ \Psi(p) = \begin{cases} (1-p^2)^3\ &\text{if}\ |p| \le 1,\\ 0\ &\text{if}\ |p| > 1.\end{cases} \]
	Note that $\Psi'(p) = -6p(1 - p^2)^2$ and $\Psi''(p) = -6(1-p^2)(1-5p^2)$ for all $p\in[-1,1]$,
	\begin{equation}\label{eq:sub}
		\max\{|\Psi'(p)|:\,p \in[-1,1]\} = \frac{96}{25\sqrt{5}} < 2,\quad\Psi'(0) = 0\quad\text{and}\quad\Psi''(0) = -6.
	\end{equation}
	
	For any convex $G\in\Con^2(\mathbb{R})\cap\mathcal{G}_0$ and $p_*,p^*\in\mathbb{R}$ with $p_* < p^*$, we will define $\t{G}_\delta\in\Con^2(\mathbb{R})\cap\mathcal{G}_0$ by setting
	\begin{equation}\label{eq:Gurb}
		\t{G}_\delta(p) =  G(p) + a\delta\Psi\left(\frac{p-p_0}{\delta}\right)\quad\text{for all $p\in\mathbb{R}$},
	\end{equation}
	with suitably chosen parameters $a\in[-1,\infty)$, $p_0\in\mathbb{R}$ and $\delta>0$. Since $G'$ is nondecreasing, it falls into (at least) one of the following cases.

	\smallskip
	
	\textbf{Case 1:} Suppose that $G'(p_*) < 0$. Let $a = -\frac{G'(p_*)}{4} > 0$. Then, there exist $p_1\in\mathbb{R}$ and $\delta>0$ such that
	$p_* = p_1-\delta < p_1+\delta \le p^*$ and $G'(p_1 + \delta) \le -2a$. Define $\t{G}_\delta$ as in \eqref{eq:Gurb} with $p_0 = p_1$. Since $\t{G}_\delta'(p) < -2a+2a =0$ on $[p_1 - \delta,p_1 + \delta]$ by \eqref{eq:sub}, both $G$ and $\t{G}_\delta$ are strictly decreasing on this interval. It follows that $\t{G}_\delta$ is quasiconvex on $\mathbb{R}$.
	
	Take any $p_2 > p_1 + \delta$ such that $G'(p_2) > 0$.
	Let
	\[ K_1 = \max\{|G'(p)|:\,p\in[p_1,p_2]\}\quad\text{and}\quad\t{K}_1(\delta) = \max\{|\t{G}_\delta'(p)|:\,p\in[p_1,p_2]\}.\]
	By \eqref{eq:sub}, $|\t{K}_1(\delta) - K_1| < 2a$ and 
	\begin{equation}\label{eq:ninin}
		\t{G}_\delta''(p_1) = G''(p_1) - \frac{6a}{\delta} < 0
	\end{equation} for sufficiently small $\delta>0$. Moreover,
	\[ \t{G}_\delta'(p_1) = G'(p_1) \le -2a < 0 < G'(p_2) = \t{G}_\delta'(p_2)\quad\text{and}\quad\t{G}_\delta''(p_2) = G''(p_2) \ge 0 \]
	by construction. Therefore,
	\begin{align}
		\t{G}_\delta''(p_1)\t{G}_\delta'(p_2)e^{-\t{K}_1(\delta)} &< \left(G''(p_1) - \frac{6a}{\delta}\right) G'(p_2)e^{-K_1 - 2a}\nonumber\\
		&<  G''(p_2)G'(p_1)e^{K_1 + 2a} \le \t{G}_\delta''(p_2)\t{G}_\delta'(p_1)e^{\t{K}_1(\delta)}\label{eq:alides},
	\end{align}
	where the first inequality in \eqref{eq:alides} holds for sufficiently small $\delta>0$. Hence, $\t{G}_\delta$ satisfies the conditions in Theorem \ref{thm:hicaz}, and the desired conclusion follows. See Figure \ref{fig:case1} for an example.
	
	\begin{figure}
		\centering
		\includegraphics[width=0.6\linewidth]{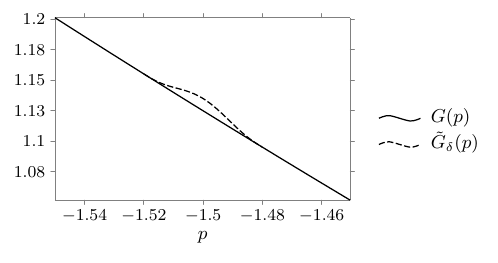}
		\caption{To illustrate Case 1 in the proof of Corollary \ref{cor:modify}, consider $G(p) = \frac12p^2$ and $[p_*,p^*] = [-2,-1]$. Then, $a = \frac12$ and $p_0 = p_1 = -\frac32$. If $p_2 = \frac32$, then $K_1 = \frac32$, and \eqref{eq:ninin}--\eqref{eq:alides} hold for $\delta < \frac{3}{1+e^5}$. In the plot above, we have taken $\delta = \frac{1}{50}$.}
		\label{fig:case1}
	\end{figure}
	
	\smallskip
	
	\textbf{Case 2:} Suppose that $G'(p^*) > 0$. Applying the change of variables in Remark \ref{rem:ters} puts us in the previous case.
	
	\smallskip
	
	\textbf{Case 3:} Suppose that $G'(p) = 0$ for all $p\in[p_*,p^*]$. Let $a=-1$, $p_0 = \frac{p_* + p^*}{2}$ and $\delta = \frac{p^* - p_*}{2}$, so that $[p_*,p^*] = [p_0-\delta,p_0+\delta]$. Define $\t{G}_\delta$ as in \eqref{eq:Gurb}. It is easy to see that $\t{G}_\delta$ is quasiconvex on $\mathbb{R}$.
	
	Take $p_1 = p_0 - \frac{\delta}2$ and $p_2 = p_0 + \frac{\delta}2$. Note that
	\[ \t{G}_\delta'(p_1) = -\Psi'\left(-\frac12\right) < 0 < -\Psi'\left(\frac12\right) = \t{G}_\delta'(p_2)\quad\text{and}\quad\t{G}_\delta''(p_1) = \t{G}_\delta''(p_2) = -\frac1{\delta}\Psi''\left(\frac12\right) < 0. \]
	Hence, $\t{G}_\delta$ satisfies the conditions in Theorem \ref{thm:hicaz} (see Remark \ref{rem:ters}), and the desired conclusion follows. See Figure \ref{fig:case3} for an example.\qedhere
	
	\begin{figure}
		\centering
		\includegraphics[width=0.6\linewidth]{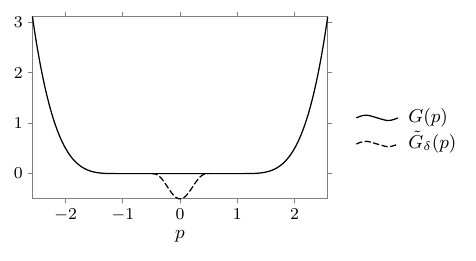}
		\caption{To illustrate Case 3 in the proof of Corollary \ref{cor:modify}, consider $G(p)=\frac{1}{2}((|p|\vee 1) - 1)^4$ and $[p_*,p^*] = [-\frac12,\frac12]$. Then, $p_0 = 0$, $\delta = \frac12$, $p_1 = -\frac14$ and $p_2 = \frac14$.}
		\label{fig:case3}
	\end{figure}

\end{proof}

\section{Extension to higher dimensions}\label{sec:multi}

Recall the outline we have given in Section \ref{ss:higher}. Fix any $G_1\in\Con^2(\mathbb{R})\cap\mathcal{G}_1$ such that
\begin{align*}
	&G_1(0) = 0,\quad G_1(p) = G_1(-p) > 0\ \ \text{and}\ \ G_1'(p) > 0\ \ \text{for all $p>0$},\\%\label{eq:mueven}\\
	&G_1''(p) < 0\ \ \text{for some $p> 0$},\ \ \text{and}\\%\label{eq:mucurve}\\
	&M := -\inf\left\{\frac{G_1''(p)}{(G_1'(p))^2}:\, p>0\right\} \in (0,\infty).%\label{eq:em}
\end{align*}
In particular, $G_1$ is an even quasiconvex function that satisfies the conditions in Theorem \ref{thm:hicaz}.

\begin{lemma}\label{lem:qbox}
	Given $G_1$ as above and any $R>0$, there exists an even convex $\breve{G}\in \Con^2(\mathbb{R})\cap\mathcal{G}_1$ such that
	the sets
	\[S_r = \left\{p\in\mathbb{R}^d:\, G_1(p_1) +
	\sum_{i=2}^d\breve{G}(p_i) \le r\right\},\quad r\le R,\] are convex.
\end{lemma}

\begin{proof}
	Suppose that $\breve{G}\in \Con^2(\mathbb{R})\cap\mathcal{G}_1$ is nonnegative, convex, and $\breve{G}(p)=0$ if and only if $p=0$. Then, $S_0 = \{(0,0,\ldots,0)\}$. For $r>0$, consider the level surface
	defined by the equation 
	\[ G_1(p_1)+\sum_{i=2}^d\breve{G}(p_i)= r. \]
	Since $G_1$ is even, it is enough to analyze only the part of this
	level surface that lies in the upper half-space
	$\mathbb{R}^d_+:=\{p\in\mathbb{R}^d:\,p_1> 0\}$. Recall that
	$G_1'(p_1)>0$ for $p_1>0$. We will consider $p_1$ as a function of
	$p':=(p_2,p_3,\ldots,p_d)$ and denote by $\pd_i$ the partial
	derivative with respect to $p_i$, $i=2,3,\ldots,d$. With this notation,
	\begin{align*}
		&G_1'(p_1)\pd_i{p_1}+\breve{G}'(p_i)=0\quad\text{and}\quad \pd_i{p_1}=-\frac{\breve{G}'(p_i)}{G_1'(p_1)}\quad\text{for all $i=2,3,\ldots,d$,}\\
		&G_1''(p_1)\pd_i{p_1}\pd_j{p_1} + G_1'(p_1)\pd_{ij}^2{p_1}+\breve{G}''(p_i)\delta_{ij} = 0\ \ \text{and}\\
		-&G_1'(p_1)\pd_{ij}^2{p_1} = \breve{G}''(p_i)\delta_{ij} + \frac{G_1''(p_1)}{(G_1'(p_1))^2}\breve{G}'(p_i)\breve{G}'(p_j)\quad\text{for all $i,j=2,3,\ldots,d$,}
	\end{align*}
	where $\delta_{ij}$ is the Kronecker delta. For every $z' = (z_2,z_3,\ldots,z_d)\in\mathbb{R}^{d-1}$,
	\begin{align}
		-G_1'(p_1)\sum_{i,j=2}^d\left(\pd_{ij}^2{p_1}\right) z_iz_j &= \sum_{i=2}^d\breve{G}''(p_i)z_i^2 + \frac{G_1''(p_1)}{(G_1'(p_1))^2}\left(\sum_{i=2}^d\breve{G}'(p_i)z_i\right)^2\nonumber\\
		&\ge \sum_{i=2}^d\breve{G}''(p_i)z_i^2 - M\left(\sum_{i=2}^d\breve{G}'(p_i)z_i\right)^2\nonumber\\
		&\ge \sum_{i=2}^d\breve{G}''(p_i)z_i^2 - M(d-1)\sum_{i=2}^d(\breve{G}'(p_i))^2z_i^2\nonumber\\
		&= \sum_{i=2}^d\left[\breve{G}''(p_i) - M(d-1)(\breve{G}'(p_i))^2\right]z_i^2\label{eq:secder}
	\end{align}
	by the Cauchy-Schwarz inequality. We will show that $\breve{G}$ can be chosen in such a way that the expression in square brackets in \eqref{eq:secder} is strictly positive on $\mathring{S}'_r:=\{p'\in\mathbb{R}^{d-1}:\, \sum_{i=2}^d\breve{G}(p_i)<r\}$.
	
	Define $J:[0,1)\to[0,\infty)$ by setting
	\begin{align}
		J(p) &= -\frac{\log(1-p) + p}{M(d-1)},\quad p\in [0,1).\label{eq:jay}\\
		\shortintertext{Note that $J(0) = 0$, $J'(0) = 0$,}
		J'(p) &= \frac1{M(d-1)}\left(\frac1{1-p} - 1\right)> 0\quad \text{for all}\ p\in(0,1),\quad \text{and}\nonumber\\
		J''(p) &= \frac1{M(d-1)}\left(\frac1{1-p}\right)^2\quad\text{for all}\ p\in[0,1).\nonumber\\
		\shortintertext{Hence,}
		J''(p) &> M(d-1)(J'(p))^2\quad\text{for all}\ p\in[0,1).\label{eq:LBem}
	\end{align}
	
	The function $J:[0,1) \to [0,\infty)$ is invertible. For any $R>0$, let $p_R = J^{-1}(R)\in (0,1)$, and define
	\begin{equation}\label{eq:deftG}
		\breve{G}(p) = \begin{cases} J(|p|) &\ \text{if}\ |p| \le p_R,\\
			R + J'(p_R)(|p| - p_R) + \frac12J''(p_R)(|p| - p_R)^2 &\ \text{if}\ |p| > p_R.
		\end{cases}
	\end{equation}
	Observe that $\breve{G}\in\Con^2(\mathbb{R})\cap\mathcal{G}_1$ is even, strictly convex, and
	$\breve{G}(0)=0$. 
	
	For every $r \le R$, we have the inclusion $\mathring{S}'_r\subset[-p_R,p_R]^{d-1}$. Therefore, the Hessian of
	$p_1$ on $\mathring{S}'_r$ is negative definite by \eqref{eq:secder} and \eqref{eq:LBem}--\eqref{eq:deftG}, and $p_1$ is concave on $\mathring{S}'_r$.
	%Moreover, the vector $(G_1'(p_1),\breve{G}'(p_2),\ldots,\breve{G}'(p_d))$ is normal to the level surface in \eqref{eq:level}, and $G_1'(p_1) \ge 0$ for $p_1\ge0$.
	It follows that the sublevel sets $S_r$ of the function $G_1(p_1) + \sum_{i=2}^d\breve{G}(p_i)$ are convex for all $r\le R$.
\end{proof}

\begin{proof}[Proof of Theorem \ref{thm:multi}]
	Fix any $G_1\in\Con^2(\mathbb{R})\cap\mathcal{G}_1$ satisfying the conditions of Lemma~\ref{lem:qbox},
	and pick a $1$-periodic $V_1\in\Lip(\mathbb{R})$ as in Theorem \ref{thm:hicaz} such that the corresponding
	effective Hamiltonian $\ol{H}_1$ is not quasiconvex on some
	interval $[\theta_0-c,\theta_0+c]$, $c\in(0,1)$, i.e., 
	\begin{equation}\label{eq:failq}
		\ol{H}_1(\lambda(\theta_0-c)+(1-\lambda)(\theta_0+c))>\max\{\ol{H}_1(\theta_0-c),\ol{H}_1(\theta_0+c)\}\quad\text{for some }\lambda\in(0,1).
	\end{equation}
	Let
	\begin{align*}
		R_1 &= \max\{G_1(\theta_0-c),G_1(\theta_0+c)\} + 2\max\{|V_1(x)|:\,x\in[0,1]\}.\\
		\shortintertext{Similarly, for any $1$-periodic $\breve{V}\in\Lip(\mathbb{R})$, let}
	    \breve{R} &= J(c) + 2\max\{|\breve{V}(x)|:\,x\in[0,1]\}
	\end{align*}
	with $J(c)$ given by \eqref{eq:jay}. Set
	\[ R = R_1 + (d-1)\breve{R}, \] 
	and define $\breve{G}$ as in \eqref{eq:deftG}. Note that $\breve{G}(\pm c) = J(c) < R$.
	
	Recall the $1$-periodic functions $f_{\theta},\breve{f}_{\theta}\in\Con^1(\mathbb{R})$ satisfying \eqref{eq:mado}--\eqref{eq:ams} as well as the functions $v_{\theta},\breve{v}_{\theta}$ defined by \eqref{eq:vees}. It is easy to check that\footnote{See the proof of Proposition \ref{prop:joe} as well as Remark \ref{rem:AT} in Appendix \ref{app:1d}.}
	\[ \max_{\theta\in[\theta_0-c,\theta_0+c]}\max_{x\in[0,1]}G_1(f_\theta(x)) \le R_1\quad\text{and}\quad\max_{\theta\in[-c,c]}\max_{x\in[0,1]}\breve{G}(\breve{f}_{\theta}(x)) \le \breve{R}. \]
	Therefore, for every $\theta = (\theta_1,\theta_2,\ldots,\theta_d)\in[\theta_0-c,\theta_0+c]\times[-c,c]^{d-1}$ and $(t,x)\in[0,\infty)\times\mathbb{R}^d$, the function $u(t,x) := v_{\theta_1}(t,x_1)+\sum_{i=2}^d\breve{v}_{\theta_i}(t,x_i)$ satisfies
	\begin{equation}\label{eq:grad}
		D_xu\in S_R.
	\end{equation}

	For every $p\in S_R$ and $x\in\mathbb{R}^d$, let 
	\[ G_R(p)=G_1(p_1)+\sum_{i=2}^d\breve{G}(p_i)\quad\text{and}\quad V(x)=V_1(x_1) + \sum_{i=2}^d \breve{V}(x_i). \]
	Then, $G_R$ is quasiconvex by Lemma \ref{lem:qbox}. %and $V\in\Lip(\mathbb{R}^d)$.
	Extend $G_R$ to a quasiconvex and superlinear $G\in\Con^2(\mathbb{R}^d)$ so that the Hamiltonian $H$ defined by $H(p,x)=G(p)+V(x)$ satisfies \eqref{eq:Hkral1}--\eqref{eq:Hkral3}. It follows from the gradient bound \eqref{eq:grad} and our argument in Section \ref{ss:higher} that the effective Hamiltonian $\ol{H}$ in \eqref{eq:effHJint} arising from the homogenization of \eqref{eq:HJint} (with $\sigma = 1$) satisfies 
	\[ \ol{H}(\theta) = \ol{H}_1(\theta_1)+\sum_{i=2}^d \ol{{\mathcal H}}(\theta_i)\quad\text{for all}\ \theta\in[\theta_0-c,\theta_0+c]\times[-c,c]^{d-1}.\]
	Hence, $\ol{H}$ is not quasiconvex on the line segment $[\theta_0-c,\theta_0+c]\times\{0\}^{d-1}$ by \eqref{eq:failq}.
\end{proof}

\section*{Acknowledgments}

Atilla Y\i lmaz thanks Govind Menon and Kavita Ramanan for valuable discussions that took place during a visit to Brown University in May 2023.

\section*{Appendices}

\appendices

\section{Periodic homogenization of viscous HJ equations with superlinear Hamiltonians}\label{app:super}

Consider the HJ equation
\begin{equation}\label{eq:HJapp}
	\pd_tu = \Delta_x u + H(D_xu,x),\quad(t,x)\in(0,\infty)\times\mathbb{R}^d,
\end{equation}
which is \eqref{eq:HJint} with $\sigma = 1$ and $\epsilon = 1$. Assume that $H:\mathbb{R}^d\times\mathbb{R}^d\to\mathbb{R}$ satisfies the following conditions for some $\alpha_0,\beta_0$ and $\eta>1$:
\begin{align}
	&\alpha_0|p|^\eta - \frac1{\alpha_0} \le H(p,x) \le \beta_0(|p|^\eta + 1)\quad\text{for all $p,x\in\mathbb{R}^d$},\label{eq:Hkral1}\\
	&|H(p,x) - H(p,y)| \le \beta_0(|p|^\eta + 1)|x-y|\quad\text{for all $p,x,y\in\mathbb{R}^d$, and}\label{eq:Hkral2}\\
	&|H(p,x) - H(q,x)| \le \beta_0(|p| + |q| + 1)^{\eta - 1}|p-q|\quad\text{for all $p,q,x\in\mathbb{R}^d$}.\label{eq:Hkral3}
\end{align}
Assume in addition that, for every $p\in\mathbb{R}^d$, the mapping $x = (x_1,\ldots,x_d)\mapsto H(p,x)$ is $[0,1]^d$-periodic, i.e., it is $1$-periodic in $x_i$ for each $i=1,\ldots,d$.

To the best of our knowledge, the homogenization of \eqref{eq:HJapp} (i.e., of \eqref{eq:HJint} with $\sigma = 1$ and $\epsilon\downarrow 0$) under this set of assumptions is not explicitly stated anywhere in the literature. However, it follows by putting together various known results and estimates, as we show below for the sake of completeness.

By \cite[Theorem 2.8]{DK17}, the Cauchy problem for \eqref{eq:HJapp} is well-posed in $\UC([0,\infty)\times\mathbb{R}^d)$. For any $\theta\in\mathbb{R}^d$, let $u_\theta$ be the unique viscosity solution of \eqref{eq:HJapp} with the initial condition $u_\theta(0,x) = \theta\cdot x$. Again by \cite[Theorem 2.8]{DK17}, we know that $u_\theta\in\Lip([0,\infty)\times\mathbb{R}^d)$. Denote its Lipschitz constant by $\kappa(\theta)$. %It follows that $|D_xu_\theta| \le \kappa(\theta)$.

Define $\t{H}:\mathbb{R}^d\times\mathbb{R}^d\to\mathbb{R}$ by
\[ \t{H}(p,x) = H(p,x) \wedge \left(|p| + \beta_0(\kappa(\theta)^\eta + 1)\right). \]
Note that $\t{H}\in\Lip(\mathbb{R}^d\times\mathbb{R}^d)$, $\displaystyle{\lim_{|p|\to\infty}\t{H}(p,x) = \infty}$ (uniformly in $x\in\mathbb{R}^d$), the mapping $x\mapsto \t{H}(p,x)$ is $[0,1]^d$-periodic, and $\t{H} = H$ on $\ol{B}(0,\kappa(\theta))\times\mathbb{R}^d$, where $B(x_0,r)$ is the ball centered at $x_0$ and with radius $r > 0$.
Therefore, $u_\theta$ is a viscosity solution of 
\begin{equation}\label{eq:Hata}
	\pd_tu = \Delta_x u + \t{H}(D_xu,x),\quad(t,x)\in(0,\infty)\times\mathbb{R}^d.
\end{equation}

Let $\Con^{i,\gamma}(\mathbb{R}^d)$, $i=0,1,2$, denote the sets of functions in $\Con^i(\mathbb{R}^d)$ whose $i^{\text{th}}$-order
derivatives are H\"{o}lder continuous with H\"{o}lder exponent $\gamma>0$. By \cite[Lemma 4.1]{Evans92}, there is a unique $\lambda(\theta)\in\mathbb{R}$ and a $[0,1]^d$-periodic $F_\theta\in\Con^{1,\gamma}(\mathbb{R}^d)$ that solves
\[ \Delta F_\theta + \t{H}(\theta + DF_\theta,x) = \lambda(\theta),\quad x\in\mathbb{R}^d, \]
in the viscosity sense. Let $K_\theta = \max\{|F_\theta(x)|:\,x\in[0,1]^d\}$, and define $\ul{v}_\theta$ and $\ol{v}_\theta$ by
\[ \ul{v}_\theta(t,x) = \lambda(\theta)t + \theta\cdot x + F_\theta(x) - K_\theta\quad\text{and}\quad \ol{v}_\theta(t,x) = \lambda(\theta)t + \theta\cdot x + F_\theta(x) + K_\theta, \]
respectively, for all $(t,x)\in[0,\infty)\times\mathbb{R}^d$. Note that $\ul{v}_\theta$ and $\ol{v}_\theta$ are viscosity solutions of \eqref{eq:Hata}. Moreover,
\[ \ul{v}_\theta(0,x) \le u_\theta(0,x) = \theta\cdot x \le \ol{v}_\theta(0,x) \]
for all $x\in\mathbb{R}^d$. By the comparison principle (see, e.g., \cite[Proposition 1.4]{D19}),
\[ \ul{v}_\theta(t,x) \le u_\theta(t,x) \le \ol{v}_\theta(t,x) \]
for all $(t,x)\in[0,\infty)\times\mathbb{R}^d$. Therefore,
\[ \lambda(\theta) = \lim_{t\to\infty}\frac{\ul{v}_\theta(t,0)}{t} \le \liminf_{t\to\infty}\frac{u_\theta(t,0)}{t} \le \limsup_{t\to\infty}\frac{u_\theta(t,0)}{t} \le \lim_{t\to\infty}\frac{\ol{v}_\theta(t,0)}{t} = \lambda(\theta). \]
By \cite[Theorem 4.1]{DK17}, we conclude that \eqref{eq:HJint} (with $\sigma = 1$) homogenizes, and the effective Hamiltonian $\ol{H}$ in \eqref{eq:effHJint} is given by $\ol{H}(\theta) = \lambda(\theta)$ for all $\theta\in\mathbb{R}^d$.

\section{Regularity of correctors}\label{app:class}

Assume that $H\in\Lip(\mathbb{R}^d\times\mathbb{R}^d)$, $\displaystyle{\lim_{|p|\to\infty}H(p,x) = \infty}$ (uniformly in $x\in\mathbb{R}^d$), and the mapping $x\mapsto H(p,x)$ is $[0,1]^d$-periodic. As we have already stated in Appendix \ref{app:super}, by \cite[Lemma 4.1]{Evans92}, for each $\theta\in\mathbb{R}^d$, there exists a unique $\ol{H}(\theta)\in\mathbb{R}$ for which the static HJ equation
\[ \Delta F + H(\theta + DF,x) = \ol{H}(\theta),\quad x\in\mathbb{R}^d, \]
has a $[0,1]^d$-periodic viscosity solution $F_\theta\in\Con^{1,\gamma}(\mathbb{R}^d)$ for some $\gamma>0$. In this appendix, we obtain some bounds on $\ol{H}(\theta)$ and show that $F_\theta\in\Con^{2,\gamma}(\mathbb{R}^d)$. These results are well known, and we include their proofs here for the sake of completeness.

Let
\begin{equation}\label{eq:sinir}
	L(\theta) = \min\{ H(\theta,x):\, x\in[0,1]^d\}\quad\text{and}\quad U(\theta) = \max\{ H(\theta,x):\, x\in[0,1]^d\}.
\end{equation}

\begin{proposition}\label{prop:isguc}
	$L(\theta) \le \ol{H}(\theta) \le U(\theta)$.
\end{proposition}

\begin{proof}
	Since $F_\theta$ is continuous and $[0,1]^d$-periodic, it is maximized at some $x_0\in[0,1]^d$. Define $\varphi$ by setting $\varphi(x) = F_\theta(x_0)$ for all $x\in\mathbb{R}^d$. By the definition of viscosity subsolutions,
	\[ \ol{H}(\theta) \le \Delta\varphi + H(\theta + D\varphi(x_0),x_0) = H(\theta,x_0) \le U(\theta). \]
	The other inequality is proved similarly.
\end{proof}

\begin{proposition}
	$F_\theta\in\Con^{2,\gamma}(\mathbb{R}^d)$.
\end{proposition}

\begin{proof}
	For every $x_0\in\mathbb{R}^d$ and $r>0$, consider the Dirichlet problem
	\begin{align}
		&\Delta F + H(\theta + DF_\theta(x),x) = \ol{H}(\theta),\quad x\in B(x_0,r),\label{eq:Poisson}\\
		&\text{$F(x) = F_\theta(x)$,\quad $x\in\partial B(x_0,r)$}\label{eq:boundary}.
	\end{align}
	%where $B(x_0,r)$ is the ball centered at $x_0$ and with radius $r > 0$.
	It follows from the definition of viscosity solutions that $F_\theta$ is a viscosity solution of \eqref{eq:Poisson}--\eqref{eq:boundary}. We can write \eqref{eq:Poisson} as $\Delta F = h_\theta$, where
	\[ h_\theta(x) = -H(\theta + DF_\theta(x),x) + \ol{H}(\theta).\]
	Since $h_\theta\in\Con^{0,\gamma}(\mathbb{R}^d)$, \eqref{eq:Poisson}--\eqref{eq:boundary} has a unique classical solution $\ol{F}_\theta\in\Con^{2,\gamma}(B(x_0,r))$, see \cite[Chapter 4]{GT_book}. Note that $w := \ol{F}_\theta - F_\theta\in\Con^{1,\gamma}(B(x_0,r))$ is a viscosity solution of $\Delta w = 0$ on $B(x_0,r)$, and $w = 0$ on $\partial B(x_0,r)$, so, in fact, $w = 0$ (by, e.g., \cite{Is95}) and $\ol{F}_\theta = F_\theta$ on $B(x_0,r)$. Since $x_0$ is arbitrary, we have the desired result.
\end{proof}

\section{Correctors in one dimension}\label{app:1d}

Assume that $H\in\Con(\mathbb{R}\times\mathbb{R})$,
\begin{align}
	&\lim_{p\to\pm\infty}H(p,x) = \infty\ \text{(uniformly in $x\in\mathbb{R}$), and}\label{eq:coer}\\
	&\text{the mapping $x\mapsto H(p,x)$ is $1$-periodic.}\label{eq:per}
\end{align}
For every $\theta\in\mathbb{R}$, let
\begin{align*}
	p_-(\theta) &= \min_{x\in[0,1]}\min\{p \in\mathbb{R}:\, H(p,x) \le U(\theta)\}\quad\text{and}\\
	p_+(\theta) &= \max_{x\in[0,1]}\max\{p \in\mathbb{R}:\, H(p,x) \le U(\theta)\},
\end{align*}
where $U(\theta) = \max\{H(\theta,x):\,x\in[0,1]\}$ as in \eqref{eq:sinir}. Note that $-\infty < p_-(\theta) \le \theta \le p_+(\theta) < \infty$.

\begin{proposition}\label{prop:joe}
	If \eqref{eq:coer}--\eqref{eq:per} are satisfied and $H\in\Lip(\mathbb{R}\times\mathbb{R})$, then for each $\theta\in\mathbb{R}$, there exists a unique $\ol{H}(\theta)\in\mathbb{R}$ and a $1$-periodic $F_\theta\in\Con^2(\mathbb{R})$ such that
	\begin{equation}\label{eq:uyku}
		F_\theta''(x) + H(\theta + F_\theta'(x),x) = \ol{H}(\theta)\quad\text{for all}\ x\in\mathbb{R}.
	\end{equation}
	Moreover,
	\begin{equation}\label{eq:nap}
		p_-(\theta) \le \theta + F_\theta'(x) \le p_+(\theta)\quad\text{for all}\ x\in\mathbb{R}.
	\end{equation}
\end{proposition}

\begin{proof}
	The first assertion follows from Appendix \ref{app:class}.
	
	We prove the second assertion. Since $F_\theta'\in\Con^1(\mathbb{R})$ is $1$-periodic, it is maximized at some $x_1\in[0,1]$, $F_\theta''(x_1) = 0$, and
	\[ H(\theta + F_\theta'(x_1),x_1) = \ol{H}(\theta) \le U(\theta) \]
	by Proposition \ref{prop:isguc}. %and $F_\theta'(x_1) \ge 0$.
	Therefore, $\theta + F_\theta'(x_1) \le p_+(\theta)$. Similarly, $F_\theta'$ is minimized at some $x_2\in[0,1]$, $F_\theta''(x_2) = 0$, and
	\[ H(\theta + F_\theta'(x_2),x_2) = \ol{H}(\theta) \le U(\theta) \]
	by Proposition \ref{prop:isguc}. %and $F_\theta'(x_2) \le 0$.
	Therefore, $\theta + F_\theta'(x_2) \ge p_-(\theta)$.
\end{proof}

The assertions in Proposition \ref{prop:joe} remain true if we replace the assumption $H\in\Lip(\mathbb{R}\times\mathbb{R})$ with a weaker one, as we state and prove below.

\begin{proposition}\label{prop:tax}
	If \eqref{eq:coer}--\eqref{eq:per} are satisfied and $H\in\Lip([-B,B]\times\mathbb{R})$ for all $B>0$, then for each $\theta\in\mathbb{R}$, there exists a unique $\ol{H}(\theta)\in\mathbb{R}$ and a $1$-periodic $F_\theta\in\Con^2(\mathbb{R})$ such that \eqref{eq:uyku} holds. Moreover, we have \eqref{eq:nap}.
\end{proposition}

\begin{proof}
	For every $\theta\in\mathbb{R}$, let
	\begin{align*}
		\t{U}(\theta) &= \max\{H(p,x):\, p\in[p_-(\theta),p_+(\theta)], x\in[0,1]\} + 1,\\
		\t{p}_-(\theta) &= \min_{x\in[0,1]}\min\{p \in\mathbb{R}:\, H(p,x) \le \t{U}(\theta)\}\quad\text{and}\\
		\t{p}_+(\theta) &= \max_{x\in[0,1]}\max\{p \in\mathbb{R}:\, H(p,x) \le \t{U}(\theta)\}.
	\end{align*}
	Note that $U(\theta) < U(\theta) + 1 \le \t{U}(\theta) \le H(\t{p}_\pm(\theta),x)$ for all $x\in\mathbb{R}$, and $[p_-(\theta),p_+(\theta)] \subset [\t{p}_-(\theta),\t{p}_+(\theta)]$.
	
	Define $\t{H}:\mathbb{R}\times\mathbb{R}\to\mathbb{R}$ by setting
	\[ \t{H}(p,x) = \begin{cases}
		H(p,x)\wedge \t{U}(\theta)&\text{if}\ p\in[\t{p}_-(\theta),\t{p}_+(\theta)],\\
		\t{U}(\theta) - (p - \t{p}_-(\theta))&\text{if}\ p \in(-\infty,\t{p}_-(\theta)),\\
		\t{U}(\theta) + (p - \t{p}_+(\theta))&\text{if}\ p \in(\t{p}_+(\theta),\infty).\\
	\end{cases}
	\]
	It follows that $\t{H}$ satisfies \eqref{eq:coer}--\eqref{eq:per}, and $\t{H}\in\Lip(\mathbb{R}\times\mathbb{R})$. Moreover,
	\begin{equation}\label{eq:inten}
		\t{H}(p,x) = H(p,x) < \t{U}(\theta)\quad\text{for all $p\in[p_-(\theta),p_+(\theta)]$ and $x\in\mathbb{R}$}.
	\end{equation}
	In particular, $\t{H}(\theta,x) = H(\theta,x)$ for all $x\in\mathbb{R}$, and
	\[ \max\{ \t{H}(\theta,x):\, x\in[0,1]\} = \max\{ H(\theta,x):\, x\in[0,1]\} = U(\theta). \]
	Furthermore, $\{p \in\mathbb{R}:\, \t{H}(p,x) \le U(\theta)\} = \{p \in\mathbb{R}:\, H(p,x) \le U(\theta)\}$ for all $x\in\mathbb{R}$, which gives
	\begin{align*}
			\min_{x\in[0,1]}\min\{p \in\mathbb{R}:\, \t{H}(p,x) \le U(\theta)\} &= p_-(\theta)\quad\text{and}\\
			\max_{x\in[0,1]}\max\{p \in\mathbb{R}:\, \t{H}(p,x) \le U(\theta)\} &= p_+(\theta).
	\end{align*}
	
	By Proposition \ref{prop:joe} (applied to $\t{H}$), there exists a unique $\ol{H}(\theta)\in\mathbb{R}$ and a $1$-periodic $F_\theta\in\Con^2(\mathbb{R})$ such that 
	\begin{equation}\label{eq:power}
		F_\theta''(x) + \t{H}(\theta + F_\theta'(x),x) = \ol{H}(\theta)\quad\text{for all}\ x\in\mathbb{R}.
	\end{equation}
	Moreover, we have \eqref{eq:nap}. Recalling \eqref{eq:inten}, we deduce that \eqref{eq:uyku} holds.
	
	It remains to prove the uniqueness of $\ol{H}(\theta)\in\mathbb{R}$. Take any $\ol{H}(\theta)\in\mathbb{R}$ and $1$-periodic $F_\theta\in\Con^2(\mathbb{R})$ such that \eqref{eq:uyku} holds. Then, we have \eqref{eq:nap} by the argument in the proof of Proposition \ref{prop:joe}. Using \eqref{eq:inten} again, we deduce that \eqref{eq:power} holds, and $\ol{H}(\theta)$ is unique by Proposition \ref{prop:joe} (applied to $\t{H}$).
\end{proof}

\begin{remark}\label{rem:AT}
	The argument that we have given above (in the proof of Proposition \ref{prop:joe}) for the bounds in \eqref{eq:nap} is one-dimensional and it only uses \eqref{eq:coer}--\eqref{eq:per}. For $d \ge 2$, such Lipschitz estimates on $F_\theta$ require a superlinear growth condition on $H$ such as \eqref{eq:Hkral1}, see \cite{AT}.
\end{remark}

\section{Proofs of Lemmas \ref{lem:order} and \ref{lem:bounds}}\label{app:lemmas}

\begin{proof}[Proof of Lemma \ref{lem:order}]
	Take any $\theta_1 < \theta_2$, and let $f_1 = f_{\theta_1}$ and $f_2 = f_{\theta_2}$. Since $\int_0^1f_i(x)dx = \theta_i$ for $i=1,2$, there exists an $x_o\in[0,1]$ such that $f_1(x_0) < f_2(x_0)$. Therefore, it suffices to prove that $f_1(x) \ne f_2(x)$ for all $x\in\mathbb{R}$. There are three possibilities:
	\begin{itemize}
		\item [(i)] Suppose $\ol{H}(\theta_1) = \ol{H}(\theta_2)$. If $f_1(x) = f_2(x)$ for some $x\in\mathbb{R}$, then $f_1(y) = f_2(y)$ for all $y \ge x$ by the uniqueness of solutions of \eqref{eq:ODE}, which is a contradiction since $f_1,f_2$ are $1$-periodic and $f_1(x_0) < f_2(x_0)$.
		\item [(ii)] Suppose $\ol{H}(\theta_1) < \ol{H}(\theta_2)$. Let
		\[ y = \inf\{x > x_0:\,f_1(x) \ge f_2(x)\}\in(x_0,\infty]. \]
		If $y<\infty$, then $f_1(y) = f_2(y)$ and $(f_2 - f_1)'(y) = \ol{H}(\theta_2) - \ol{H}(\theta_1) > 0$ by \eqref{eq:ODE}, so $f_1(y-\delta) > f_2(y-\delta)$ for sufficiently small $\delta>0$, which is a contradiction. Hence, $f_1(x) < f_2(x)$ for all $x>x_0$. Since $f_1,f_2$ are $1$-periodic, we conclude that $f_1(x) < f_2(x)$ for all $x\in\mathbb{R}$.
		\item [(iii)] Suppose $\ol{H}(\theta_1) > \ol{H}(\theta_2)$. Let
		\[ y = \sup\{x < x_0:\,f_1(x) \ge f_2(x)\}\in[-\infty,x_0). \]
		If $y>-\infty$, then $f_1(y) = f_2(y)$ and $(f_2 - f_1)'(y) = \ol{H}(\theta_2) - \ol{H}(\theta_1) < 0$ by \eqref{eq:ODE}, so $f_1(y+\delta) > f_2(y+\delta)$ for sufficiently small $\delta>0$, which is a contradiction. Hence, $f_1(x) < f_2(x)$ for all $x<x_0$. Since $f_1,f_2$ are $1$-periodic, we conclude that $f_1(x) < f_2(x)$ for all $x\in\mathbb{R}$.\qedhere
		%\item [(iii)] Suppose $\ol{H}(\theta_1) > \ol{H}(\theta_2)$. If $f_1(x) = f_2(x)$ for some $x>x_0$, then $(f_2 - f_1)'(x) = \ol{H}(\theta_2) - \ol{H}(\theta_1) < 0$ by \eqref{eq:ODE}, so $f_1(x + \delta) > f_2(x + \delta)$ for sufficiently small $\delta>0$. Let
		%\[ z = \inf\{y > x+\delta:\,f_1(y) \le f_2(y)\}\in(x+\delta,\infty]. \]
		%If $z<\infty$, then $f_1(z) = f_2(z)$ and $(f_2 - f_1)'(z) = \ol{H}(\theta_2) - \ol{H}(\theta_1) < 0$ by \eqref{eq:ODE}, so $f_1(z-\delta') < f_2(z-\delta')$ for sufficiently small $\delta'>0$, which is a contradiction. Hence, $f_1(y) > f_2(y)$ for all $y>x+\delta$. However, this is a contradiction, too, since $f_1,f_2$ are $1$-periodic and $f_1(x_0) < f_2(x_0)$.
	\end{itemize}
\end{proof}

\begin{proof}[Proof of Lemma \ref{lem:bounds}]
	We will write $K = K_1(\theta_1,\theta_2)$ for the sake of convenience. Note that
	\begin{equation}\label{eq:diffODE}
		(g_{\theta_1,\theta_2})'(x) = \left(f_{\theta_2} - f_{\theta_1}\right)'(x) = G(f_{\theta_1}(x)) - G(f_{\theta_2}(x)) + \ol{H}(\theta_2) - \ol{H}(\theta_1)
	\end{equation}
	by \eqref{eq:ODE}. We divide the proof into two cases.
	
	\smallskip
	
	\textbf{Case 1:} If $\ol{H}(\theta_2) - \ol{H}(\theta_1) \le 0$, then $(g_{\theta_1,\theta_2})'(x) \le K g_{\theta_1,\theta_2}(x)$ for all $x\in[0,1)$ by \eqref{eq:diffODE} and the mean value theorem. Take any $x_1,x_2\in[0,1)$. \begin{itemize}
		\item [(i)] If $0 \le x_1 < x_2 < 1$, then
		\[ \log\left(\frac{g_{\theta_1,\theta_2}(x_2)}{g_{\theta_1,\theta_2}(x_1)}\right) \le K(x_2 - x_1) \le K. \]
		\item [(ii)] If $0\le x_2 < x_1 < 1$, then $x_1 < x_2 + 1 < x_1 + 1$, and
		\[ \log\left(\frac{g_{\theta_1,\theta_2}(x_2)}{g_{\theta_1,\theta_2}(x_1)}\right) = \log\left(\frac{g_{\theta_1,\theta_2}(x_2 + 1)}{g_{\theta_1,\theta_2}(x_1)}\right) \le K(x_2 + 1 - x_1) \le K \]
		since $g_{\theta_1,\theta_2}$ is $1$-periodic.
	\end{itemize}
	
	\smallskip
	
	\textbf{Case 2:} If $\ol{H}(\theta_2) - \ol{H}(\theta_1) > 0$, then $(g_{\theta_1,\theta_2})'(x) \ge -K g_{\theta_1,\theta_2}(x)$ for all $x\in[0,1)$ by \eqref{eq:diffODE} and the mean value theorem. Take any $x_1,x_2\in[0,1)$.
	\begin{itemize}
		\item [(i)] If $0 \le x_1 < x_2 < 1$, then
		\[ \log\left(\frac{g_{\theta_1,\theta_2}(x_2)}{g_{\theta_1,\theta_2}(x_1)}\right) \ge -K(x_2 - x_1) \ge -K. \]
		\item [(ii)] If $0\le x_2 < x_1 < 1$, then $x_1 < x_2 + 1 < x_1 + 1$ and
		\[ \log\left(\frac{g_{\theta_1,\theta_2}(x_2)}{g_{\theta_1,\theta_2}(x_1)}\right) = \log\left(\frac{g_{\theta_1,\theta_2}(x_2 + 1)}{g_{\theta_1,\theta_2}(x_1)}\right) \ge -K(x_2 + 1 - x_1) \ge -K \]
		since $g_{\theta_1,\theta_2}$ is $1$-periodic.
	\end{itemize}
	
	\smallskip
	
	We have proved that
	\[ \log\left(\frac{g_{\theta_1,\theta_2}(x_2)}{g_{\theta_1,\theta_2}(x_1)}\right) \le K_1(\theta_1,\theta_2) \]
	for every $x_1,x_2\in\mathbb{R}$. Finally, note that $\int_0^1g_{\theta_1,\theta_2}(x)dx = \theta_2 - \theta_1$. Therefore, $g_{\theta_1,\theta_2}(x) = \theta_2 - \theta_1$ for some $x\in\mathbb{R}$ by the mean value theorem for integrals, and \eqref{eq:bounds} follows.
\end{proof}

\bibliography{viscousHJ}
\bibliographystyle{alpha}

\end{document}